\documentclass[leqno,12pt]{amsart} 
\title[Generalised nice sets]{Generalised nice sets}
\usepackage{amssymb,color, bm}
\author[C.~Draper]{Cristina Draper${}^\star$}

\author[T.~Meyer]{Thomas L. Meyer, Juana S\'anchez-Ortega${}^\dagger$}



\thanks{${}^\star$Universidad de M\'alaga, Spain, ORCID 0000-0002-2998-7473. Corresponding author: cdf@uma.es.}
\thanks{${}^\dagger$University of Cape Town, South Africa, ORCID 0000-0003-2444-8830.} 

\usepackage{amssymb, mathrsfs, amsmath, tabu, amsthm}
\usepackage{hyperref}
\usepackage{tikz-cd, graphicx} 
\usepackage{pdfpages}
\usepackage{enumitem}
\usepackage{ragged2e}
\usepackage[utf8]{inputenc}
 \usepackage{longtable}
 \usepackage[most]{tcolorbox}
  
\usetikzlibrary{decorations.markings}

\hypersetup{
	colorlinks = true,
	urlcolor = blue,
	linkcolor = blue,
	citecolor = red
}
\newcommand{\setbar}[2]{\{#1 \, \lvert \, #2\}} 
\newcommand{\bfemph}[1]{\emph{\textbf{#1}}}

\newcommand{\f}[1]{\mathfrak{#1}}

\newcommand{\bb}[1]{\mathbb{#1}}

\newcommand{\ep}{\varepsilon}

\DeclareMathOperator{\id}{id}

\DeclareMathOperator{\freq}{freq}

\theoremstyle{plain} 
\newtheorem{theorem}{Theorem}[section]
\newtheorem{lemma}[theorem]{Lemma}
\newtheorem{cor}[theorem]{Corollary}
\newtheorem{prop}[theorem]{Proposition}

\theoremstyle{definition} 
\newtheorem{define}[theorem]{Definition}

\newtheorem{remark}[theorem]{Remark}

\makeatletter
\newcommand{\tagsLeft}{\tagsleft@true\let\veqno\@@leqno}
\newcommand{\tagsRight}{\tagsleft@false\let\veqno\@@eqno}
\makeatother

\begin{document}


\maketitle 


\begin{abstract}
A new combinatorial object, called generalised nice set, is classified up to collineations of the Fano plane. 
This classification is necessary to find the graded contractions of all the exceptional complex Lie algebras of dimension at least 52, endowed with   $\mathbb Z_2^3$-gradings coming from the octonions. Our classification is of purely combinatorial nature.
 \end{abstract}

\section{Introduction} 

This work is part of an ambitious research project (\cite{Paper1,Paper2,Paper4}) that seeks to find and classify the graded contractions of various 
 $\mathbb{Z}^3_2$-gradings on all the exceptional complex Lie algebras simultaneously. 
 A graded contraction of a $G$-graded Lie algebra $L$, for $G$ a finite abelian group,
 is another Lie algebra $L^\ep$ which can be obtained from the original algebra $L$ and 
 from  a particular map $\ep\colon G\times G\to \mathbb C$. To obtain the new Lie algebra entails modifying the bracket of homogeneous elements by means of  a constant given by the map $\ep$.
  Finding the graded contractions means finding the maps $\ep$  that turn $L^\ep$ into a Lie algebra,  
    which incidentally, is usually \lq\lq more abelian\rq\rq\, than the original one.
 The notion of graded contraction was introduced by physicists in the early  1990s \cite{CPSW,otro91} as a generalisation of Wigner-In\"on\"u contractions.
There is a plethora of different concepts regarding contractions, degenerations and deformations, which the physicists study in connection with  limit theories  (see \cite[Chapter~5]{libroLiefisica} on applications of the Lie theory to physics). 
From an algebraic point of view, the graded contractions can contribute significantly to the ongoing classification of solvable Lie algebras.  This contribution tends to come about by providing unknown and  unexpected    examples,  in dimensions which pose difficulties.
 
A pair of sufficient conditions for a map $\ep$ to produce a graded contraction are given by 
$\ep_{jk}=\ep_{kj}$ and $\ep_{jk}\ep_{l,j+k}=\ep_{jl}\ep_{k,j+l}=\ep_{kl}\ep_{j,k+l}$, if $G=\{g_0,\dots, g_n\}$ and $\ep_{jk}:=\ep(g_j,g_k)$. 
(These   are also necessary conditions if the grading under study is  sufficiently symmetrical, as the ones  in \cite{Paper4}.)
In particular, the set of pairs $\{\{j,k\}:\ep_{jk}\ne0\}$, called the \emph{support} of $\ep$, satisfies   a certain absorbing property. 
We call any set of pairs satisfying such a  property (for $G=\mathbb Z_2^3$) a \emph{generalised nice set} (see Definition~\ref{defGNS}).
The term comes from the fact that they are a natural generalisation of the \emph{nice sets} considered in \cite{Paper1}, which have been proved to coincide with the supports of the graded contractions of the $ \mathbb Z_2^3$-grading $\Gamma_{\f{g}_2}$ on the complex exceptional Lie algebra $\mathfrak{g}_2$, induced by the
natural  $\mathbb{Z}_2^3$-grading on the complex octonions   in \cite{gradsO}. The classification of these nice sets up to collineations   
was the cornerstone in \cite{Paper1}  for achieving a complete classification of graded contractions of $\Gamma_{\f{g}_2}$. 
Likewise, the classification of the generalised nice sets is the first achievement we must reach in order to find the graded contractions of the above mentioned gradings on the exceptional Lie algebras $\mathfrak{f}_4, \mathfrak{e}_6, \mathfrak{e}_7$ and $\mathfrak{e}_8$. 
As an immediate application, this will be a  contribution to the classification of solvable and nilpotent Lie algebras, since our catalogue  of 
generalised nice sets will give a large collection of completely unknown algebras, of large dimensions, with varying   properties. In fact, most of them are nilpotent or at least solvable, and the nilpotency and solvability indices can be calculated directly from the support. 

Previous results of this kind, which compute the graded contractions of a fixed $G$-graded Lie algebra $L$, appear in \cite{checos06, gr-cont} for $L$ of dimension $8$. Their calculations rely heavily on the use of a computer system. In \cite{Paper1} we explore the nice sets as a tool  to classify the $\mathbb{Z}^3_2$-graded contractions of the complex exceptional Lie algebra $\mathfrak{g}_2$, of dimension 14, without computer assistance. 
Our new combinatorial object, the generalised nice set,   allows us to obtain graded contractions for algebras of dimensions as large as $52$, 78, 133 and 248 also  without computer assistance! 
This is a challenging task, let us simply note that there are 24 nice sets up to collineations, while there are 245 not collinear generalised nice sets! (see \S\ref{conclusion}).
Although we will provide our classification of generalised nice sets without computer, 
we have added  \S\ref{computer}, 
where we use computer algorithms to provide an alternative proof of
 the exact equivalence classes in the classification of generalised nice sets. The assistance these algorithms offer lies 
mainly in the verification of many technical calculations involved in describing all the distinct equivalence classes. Once we narrow the possibilities for generalised 
nice sets from the $2^{36}$ initial possibilities to a few hundred candidates, the task of then identifying exactly the distinct equivalence classes is a 
subtle affair. Computer assistance eliminates the chance of small technical errors between the many equivalence classes. 

The paper is structured as follows: 
in \S\ref{sec3} we introduce the main object of this paper: the generalised nice sets, after recalling some preliminary background on nice sets and collineations. We prove 
some general properties, which allow us to split the combinatorial classification problem into four mutually exclusive  cases, which rely on the specific forms that generalised nice sets might take. 
We tackle these cases in the following four sections: 
\S\ref{gnsXSection} deals with generalised nice sets contained in $X$,
\S\ref{TintXempty} is devoted to generalised nice sets having  empty intersection with $X$,
in \S\ref{partI} we study generalised nice sets whose intersection with $X$ is not generalised  nice, and lastly,
generalised nice sets whose non-empty intersection with $X$ is  generalised nice are considered in \S\ref{partII}.
Our main results are Corollary~\ref{somegns}, Proposition~\ref{way}, Theorem~\ref{GnsListThm} and Theorem~\ref{todo}, which give exact criteria to   determine whether a set is generalised nice. The equivalence classes up to collineations are shown  in tables   scattered throughout the work. 
A short summary is provided in 
\S\ref{conclusion}. 
As a bonus, \S\ref{computer} 
provides computer algorithms   as an alternative method to find the exact equivalence classes in the classification of generalised nice sets. 
\smallskip

Our problem can also be considered a combinatorial problem on abelian groups (see  \cite{combi6,combi1} in relation to such issues).  Generalised nice sets are  in one-to-one correspondence with subsets of $G\times G$ satisfying  a certain absorbing property (for our abelian group $G=\mathbb Z_2^3$),  
where the collinearity corresponds to the passage through an automorphism of the group. (Recall that the simple group $\mathrm{Aut}(\mathbb Z_2^3)$ is isomorphic to the group of collineations of the Fano plane.) 

Our generalised nice sets can also be considered as sets of  points and  segments (portions of lines between two points) in an \emph{extended} Fano plane. We would like to emphasize that the Fano plane,  the unique projective plane of order 2,   repeatedly appears in Mathematics  (see a very nice review of its descriptions in \cite{combi4}). 
  Combinatorial and geometrical constructions may lay dormant for some time before revealing their usefulness, but frequently do so eventually. The literature is full of examples (see, for instance, textbooks as \cite{librocombi} and references therein). 
Particularly, geometry concerning points and lines. Open problems and conjectures about many objects in finite projective spaces appear in \cite{combi5}. The work \cite{combi3} deserves special mention. It relates classical configurations of lines to some exceptional complex Lie algebras of series E. In fact, it refers to some of the gradings we are going to apply our results,   which can be considered as octonionic-gradings. 
 
	
\section{Generalised nice sets}  \label{sec3}

\subsection{Preliminaries on collineations} \label{col}

 Let $I: = \{1,\ldots, 7\}$, $I_0 := I \cup \{0\}$.   Consider the following pictorial summary of the Fano plane with its 7 lines, each consisting of 3 points in $I$:
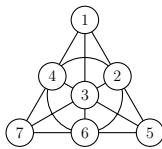
\vspace{-2pt}\begin{center}
\begin{tikzpicture}[scale=0.5, every node/.style={transform shape}] \label{fano}
\draw 
(30:1) node[circle, draw, fill=white]{2}  -- (210:2) node[circle, draw, fill=white]{7}
(150:1) node[circle, draw, fill=white]{4} -- (330:2) node[circle, draw, fill=white]{5}
(270:1) node[circle, draw, fill=white]{6} -- (90:2) node[circle, draw, fill=white]{1}
(90:2)  -- (210:2) -- (330:2) -- cycle
(0:0) node[circle, draw, fill=white]{3}   circle (1);
\end{tikzpicture}
\end{center}
That is, ${\bf L} = \{\{1, 2, 5\}, \, \{5, 6, 7\},\, \{7, 4, 1\}, \, \{1, 3, 6\}, \, \{6, 4, 2\},\, \{2, 7, 3\}, \, \{3, 4, 5\}\}$.   Note that the lines are not ordered and that any two 
different lines intersect at exactly one point.
 For $i, j \in I $, $i\ne j$,  we let $i \ast j$ denote the unique element in $I -\{i,j\}$ which lies on the same line as
$i$ and $j,$ in the above picture. Furthermore, we denote by $\ell_{ij}$ the line containing both $i$ and $j$, i.e. $\ell_{ij}=\{i,j,i\ast j \}\in {\bf L}$.
We can extend $\ast$ to an operation on $I_0$ by setting  $0\ast i=i\ast 0=i$, $0\ast 0=0$, and $i\ast i=0$, for any $i\in I$.

A collineation (also called an automorphism or symmetry) of the Fano plane is a permutation of $I$ which preserves collinearity. That is, it must carry collinear   points (points on the same line) to collinear points. In other words, a bijective map $\sigma\colon I \to I$ is called a \emph{collineation} if $\sigma(i*j) = \sigma(i)*\sigma(j)$, for all $i, j \in I$, $i\ne j$.
We write $S_\ast (I)$ for the group consisting of all the collineations of the Fano plane.
We can extend this definition, to maps defined on $I_0,$ as follows: $S_\ast (I_0)=\{\sigma\colon I_0 \to I_0\textrm{ bijective }\vert \ \sigma(i*j) = \sigma(i)*\sigma(j), \ \forall i, j \in I_0\}$.
Any $\sigma\in S_\ast (I_0)$  will map 0 onto 0,  and so the map $S_\ast (I_0)\to S_\ast (I)$ given by restricting $\sigma\mapsto\sigma\vert_I$ is a group isomorphism. This enables us to frequently abuse the notation by referring to the elements in $S_\ast (I_0)$ as collineations.

As in \cite[Definitions 3.9 and 3.12]{Paper1}, we call pairwise distinct elements $i, j, k \in I$ \emph{generative}, if $k\neq i\ast j$. 
We may also refer to the whole triplet $\{i, j, k\}$ as \emph{generative} because the definition is independent of the order of the elements $i, j,$ and $k$.  
Collineations preserve generative triplets: if $i, j, k \in I$ are generative then $\sigma(i), \sigma(j), \sigma(k)$ are generative, for all $\sigma \in S_\ast(I)$. Furthermore, any two generative triplets $\{i, j, k\}$ and $\{i', j', k'\}$ yield a unique $\sigma\in S_\ast(I)$ such that $\sigma(i) = i'$, $\sigma(j) = j'$, $\sigma(k) = k'$. 
In particular,  if $i, j, k \in I$ are generative, setting $\sigma(1)=i$, $\sigma(2)=j,$ and $\sigma(3)=k$ uniquely determines $\sigma\in S_\ast(I)$ since the remaining values of 
$\sigma$ are forced by collineations' preservation of $\ast$. This unique collineation will be denoted by $\sigma_{ijk}$. In this way, the whole group $  S_\ast(I)=\{\sigma_{ijk}:i,j,k\in I \,\textrm{generative} \}$ has $7\cdot 6\cdot 4=168$ elements;
 we can choose $i\in I$ arbitrarily, $j\ne i$, and $k\ne i,j,i*j$. This group is well-known to be $\mathrm{PGL}(3,2)$ \cite[p.~131]{libroFano}.

 
\subsection{Preliminaries on nice sets}

The work \cite{Paper1} classified admissible graded contractions of a fine $\mathbb Z_2^3$-grading, on the complex simple Lie algebra $\mathfrak{g}_2$, up to certain equivalence relations. 
 To solve this problem, the authors adopted a combinatorial approach
involving finding the so-called \emph{nice sets}. These are the subsets of the set of edges of the Fano plane, $X := \{ \{i, j\} \mid i, j \in I, \, i \neq j \},$ which satisfy a certain absorbing property.

\begin{define} \cite[Definition 3.9]{Paper1}
A subset $T$ of $X$ is called {\bf nice}  if   the presence of $\{i, j\}, \{i\ast j, k\} \in T$, for some generative $i, j, k \in I$, implies  that the set 
\begin{equation}\label{eq_defP}
P_{\{i,j,k\}} := \{\{i, j\}, \{j,k\},  \{k,i\}, \{i, j\ast k\}, \{j, k\ast i\}, \{k, i\ast j\}\}
\end{equation}
 is fully contained in $T$.  
\end{define}

 Notice that $P_{\{i,j,k\}}$ (for generative $i,j,k\in I$) has cardinal 6, is itself  a nice set, and is independent of the order chosen of the indices $i$,
 $j$ and $k$.
For any line $\ell\in\mathbf{L}$ and for any $i\in I$, some more examples of nice sets are:
\begin{itemize}
\item[-]  $X_{\ell}:=\{\{i,j\}\in X:i, j\in\ell\}$;  
\item[-]  $X_{\ell^C}:=\{\{i,j\}\in X: i,j\notin\ell\}$; 
\item[-]  $X_{(i)}:=\{\{i,j\}\in X:j\ne i\}$;
\item[-]  $X^{(i)}:=\{\{j,k\}\in X:j*k=i\}$.  
\end{itemize}
  If $\{i, j, k\}$ is a generative triplet, the following set of cardinal 10 is also nice,
\begin{equation}\label{eq_losTes}
\begin{array}{ll}
T_{(i,j,k)}:&= P_{\{i,j,k\}}\cup \big\{\{ i,i*j \},\{ i,i*k \},\{ i*j,i*k \},\{ i,i*j*k \}\big \}\\
&=P_{\{i, j, k\}} \cup P_{\{i, j, \,i \ast k\}} \cup P_{\{i,  \,i \ast j, k\}} 
\\&=X_{(i)}\cup \big\{\{ i*k,i*j \},\{ j,i*k \},\{ j,k \},\{ k,i*j \}\big \}.
\end{array}
\end{equation}  
Note that definition of $T_{(i,j,k)}$ does depend on the order of its indices. In particular, the first index $i$ plays a different role from the other two indices, $j$ and $k$.
The complete list of nice sets can be extracted from \cite[Propositions~3.23 and 3.25]{Paper1}:
 
\begin{theorem}  \cite[Theorem~3.9]{Paper2} 
The nice sets are exactly: $X$,  $X\setminus  X_{\ell^C}$, $P_{\{i,j,k\}}$, $T_{(i,j,k)}$, $X^{(i)},$ 
 and any subset of $X_{\ell}$, $X_{\ell^C}$, or $X_{(i)}$,   for some $\ell\in\mathbf{L}$, $i\ne j\in I$, $k\notin\ell_{ij}$.
\end{theorem}

According to \cite[Theorem~3.27]{Paper1}, there are exactly 779 nice sets. Each nice set induces a new $\mathbb Z_2^3$-graded Lie algebra obtained by graded contraction of $\mathfrak g_2$ \cite{Paper1}, as well as two more Lie algebras obtained by graded contraction of the orthogonal Lie algebras $\mathfrak{so}(7,\mathbb C)$ and $\mathfrak{so}(8,\mathbb C)$  (see Definitions~3.1 and  3.10 in \cite{Paper2}).  However, in all three cases, the Lie algebras related to \emph{collinear} nice sets are necessarily isomorphic.

\begin{remark}\label{re_colineal} \cite[Definition 3.16]{Paper1}
 A natural action of the group $S_\ast(I)$ on $X$ is given by $\widetilde{\sigma}(\{i, j\}) = \{\sigma(i), \sigma(j)\}$, for $\sigma \in S_\ast(I)$ and $\{i, j\} \in X$. 
Hence, there is a natural action of $S_\ast(I)$ on $\mathcal{P}(X)$.
This gives rise to an equivalence relation on the set of all nice sets: $T \sim_c T'$ if there exists $\sigma \in S_\ast(I)$ such $\widetilde{\sigma}(T) = T'$. In such  cases, we say that $T$ and $T'$ are \bfemph{collinear}.  \end{remark}

 The classification of nice sets up to collineations is made interesting by the combination of the facts that any nice set induces a Lie algebra, and that collinear nice sets induce isomorphic Lie algebras. This classification is achieved in \cite[Theorem~3.27]{Paper1}, 
according to which there exist exactly 24 classes of nice sets up to collineations.  Moreover, three of these classes give rise to infinite families of non-isomorphic Lie algebras  depending on parameters \cite[\S4]{Paper1}.


\subsection{Generalised nice sets}

 In a work which is still in progress \cite{Paper4}, we have proved that a new type of structure (strongly inspired by the above nice sets)
 must serve as the support of any graded 
contraction of the $\mathbb Z_2^3$-gradings, induced from the complex octonion algebra, on the four exceptional complex Lie algebras other than $\mathfrak g_2$.
A philosophy similar to that of \cite{Paper2}, will offer interesting new families of high-dimensional solvable Lie algebras. This is the main motivation behind the definition and investigation of the following sets.\smallskip



Denote by $X_0 := \{ \{i, j\} \mid i, j \in I_0\}$, a set of cardinal 36 containing   all 21 of \emph{edges} in $X$.
The extension of notation from nice sets to the current context introduces some ambiguity. In particular, $X_0$ is not a subset of $\mathcal P(I_0)$, 
since $\{i, j\}$ does not refer to a subset of $I_0$ but instead to an unordered set with two (not necessarily distinct) elements in $I_0$. To clarify,
 the \emph{pairs}
$\{i, i\}$ are in $X_0$. Denote also by $X_E := \big\{\{i, i\} \mid i \in I_0 \big\}$ and  $X_F := \big\{\{0, i\} \mid i \in I_0 \big\}$, 
  so that $X_0=X\cup X_E\cup X_F$. 
  We also continue to use the notation in \eqref{eq_defP} for $P_{\{i,j,k\}}$, this time using arbitrary indices $i, j, k \in I_0$ instead of generative $i, j, k \in I$.

\begin{remark}\label{re_tipos}
 It is useful to have the different possible forms of the subsets $P_{\{i,j,k\}}$ written explicitly.     
  To be precise, if $i$ and $j$ in $I$ are   distinct,
 \begin{itemize}
 \item[-] $P_{\{ 0,0,0\}}=  \big\{ \{0,0 \} \big\};$
 \item[-] $P_{\{ 0,0,i\}}=  \big\{ \{ 0,i\},\{ 0,0\} \big\};$
 \item[-] $P_{\{ i,i,i\}}=  \big\{ \{ 0,i\},\{ i,i\} \big\};$
 \item[-] $P_{\{ 0,i,i\}}=  \big\{ \{ 0,i\},\{i,i \},\{ 0,0\} \big\}$;
 \item[-] $P_{\{i,i,j \}}=  \big\{ \{i,i \},\{i,j \}, \{j,0 \}, \{i,i\ast j \} \big\};$
 \item[-] $P_{\{ 0,i,j\}}=  \big\{ \{0,i \},\{ 0,j\},\{ 0,i\ast j\},\{ i,j\} \big\};$
 \item[-] $P_{\{i,j,i\ast j \}}=  \big\{ \{ i,j\},\{ j,i\ast j\},\{ i,i\ast j\},\{i,i\},\{j,j \},\{i\ast j,i\ast j \} \big\}.$
 \end{itemize}
 Observe that now the cardinals are, respectively, $1, 2,2,3,4,4$, and 6, in contrast with the case where $i,j,k\in I$ are generative
  (where the cardinal was always 6).
\end{remark}

\begin{define}\label{defGNS}
A subset $T$ of $X_0$ is said to be a {\bf generalised nice set} if $\, \{i, j\},\, \{i\ast j, k\} \in T$ implies $P_{\{i,j,k\}} \subseteq T$, for any $i, j, k \in I_0$.
\end{define}
 
The term we have assigned to it, \emph{generalised},  
may cause confusion because a nice subset  may not be generalised.
 If $T\subseteq X$ is a generalised nice set then $T$ is nice, but the nice set $X$ is not a generalised nice set. 
 The justification for the use of this term is formal similitude. 
\smallskip

We begin by noticing some trivial facts.

\begin{lemma} \label{trivialfacts}
The following hold for a generalised nice set $T \subseteq X_0$ and $i, j \in I$ distinct.
\begin{itemize}
\item[\rm (i)] $T \cap X$ is a nice set;
\item[\rm (ii)] if $\{i, 0\} \in T$, then $\{0, 0\} \in T$; 
\item[\rm (iii)] if $\{i, i \ast j\}$, $\{j, i \ast j\} \in T$, then $\{i \ast j, i \ast j\} \in T$;
\item[\rm (iv)] if $T \subseteq X$, then $T \cup \big\{\{0, 0\}\big\}$ is generalised nice.

\end{itemize}
\end{lemma}

\begin{proof} 
The proof of (i) is trivial.
If $\{i, 0\} \in T$, 
then  $\{0,0\}\in P_{\{i,0,0\}} \subseteq T$, so (ii) is clear.
A similar argument works for (iii) since $\{i, i \ast j\},\{j, i \ast j\} \in T$ implies $P_{\{i, i \ast j, i \ast j\}} \subseteq T$. 
For (iv), assume $\, \{i, j\},\, \{i\ast j, k\} \in T\cup \big\{\{0, 0\}\big\}$,
and let us check that then
$P_{\{i,j,k\}} \subseteq T\cup \big\{\{0, 0\}\big\}$. 
If $i=j=0$, then  $\{0,k \} \in T\cup \big\{\{0, 0\}\big\}$ means $k=0$ and $\big\{\{0,0\}\big\}=P_{\{0,0,0\}} \subseteq T\cup \big\{\{0, 0\}\big\}$.
If this is not the case but $i\ast j=k=0$, we have $i=j$ and so  the  contradiction $\{i, i\} \in T \subseteq X$ with $i\ne0$.
Finally, if both $\, \{i, j\},\, \{i\ast j, k\} \in T$, it is enough to take into account that $T$ is generalised nice.
\end{proof}

\begin{remark} \label{re_PnoesGNS}
 In contrast to the context of nice sets, we see that sets of the form $P_{\{i,j,k\}}$ are not necessarily generalised nice. For instance, 
$P_{\{i, i, i\}}\ (i\neq 0)$  is not, since it does not satisfy  Lemma~\ref{trivialfacts}(ii). 
  In what follows we write $\langle S \rangle$ to denote the smallest generalised nice set containing $S\subseteq X_0$.
It is easy to check that $\langle P_{\{i, i, i\}} \rangle = P_{\{0, i, i\}}$. 
Similarly,  $P_{\{i, i, j\}} $ is not generalised nice, and 
\begin{equation}\label{eq_auxiliar}
\langle P_{\{i, j, i\}} \rangle = \big\{\{0, 0\}, \{0, i\}, \{0, j\}, \{0, i\ast j\}, \{i, i\}, \{i, j\}, \{i, i\ast j\} \big\}.
\end{equation}
Indeed, if $T$ is any generalised nice set   containing $P_{\{i, j, i\}} $, then  $\{i,i\ast j\}$ and $ \{j,0\}$   belong  to $T$, and so 
$\{0, i\}, \{0, i\ast j\}\in P_{\{i,i*j,0\}}\subseteq T$.  
Recall that $\{0, 0\}\in   T$ by Lemma~\ref{trivialfacts}(ii). 
 Thus, the set on the right side of
\eqref{eq_auxiliar} is contained in $T$. We are finished because that set is already generalised nice.
\end{remark}
\smallskip

   The notion of collinearity   recalled in Remark~\ref{re_colineal} can be extended to our setting without changes: 
   $S_\ast(I_0)$ acts on $ X_0$ 
   by $\widetilde{\sigma}(\{i, j\}) = \{\sigma(i), \sigma(j)\}$, for $\sigma \in S_\ast(I_0)$ and $\{i, j\} \in X_0$. Therefore,  $S_\ast(I_0)$ acts on $\mathcal{P}(X_0)$ 
   too. Two generalised nice sets $T$ and $T'$ are \emph{collinear}, denoted $T \sim_c T'$,  if there exists $\sigma \in S_\ast(I_0)$ such $\widetilde{\sigma}(T) = T'$.

 The main aim of this paper is to achieve a classification of generalised nice sets up to collineation.
 Notice that any generalised nice set $T \subseteq X_0$ can be expressed as $T = (T \cap X) \cup (T - X)$, where $T \cap X$ is a nice set by Lemma~\ref{trivialfacts}(i). In order to make this combinatorial problem more manageable we  will split our study of generalised nice sets into the following four cases:

\begin{enumerate}
\item[-] $T \subseteq X;$
\item[-] $T \cap X = \emptyset;$
\item[-]    $T \cap X$ is not generalised nice; 
\item[-] $T \cap X \neq \emptyset$ and $T \cap X$ is generalised nice. 
\end{enumerate} 

\smallskip 

We close this section with a result relating the parts $T - X$ and $T \cap X$ from the previous decomposition of $T$:

\begin{prop} \label{charTintXnotgns}
The following are equivalent for a generalised nice set $T$ satisfying that both  $T - X$
  and $T \cap X$ are non-empty: 
\begin{itemize}
\item[\rm (i)] $T - X$ is not a generalised nice set.
\item[\rm (ii)] There exist $i, j \in I$ distinct such that $\langle P_{\{i, j, i\}} \rangle \subseteq T$.
\item[\rm (iii)] There exist $i, j \in I$ distinct such that $\{i, i\}, \{0, j\} \in T$.
\item[\rm (iv)] $T \cap X$ is not a generalised nice set.
\end{itemize}
\end{prop}

\begin{proof}

\noindent (i) $\Rightarrow$ (ii). 
There exist $a, b, c \in I_0$ such that $\{a, b\}, \{a \ast b, c\} \in T - X$ but $P_{\{a, b, c\}} \nsubseteq T - X$. 
As $P_{\{a, b, c\}} \subseteq T$, this means that $P_{\{a, b, c\}} \cap X\ne\emptyset$.   Since
$\{a, b\}, \{a \ast b, c\} \in \big\{\{i, i\} \mid i \in I_0 \big\}\cup \big\{\{i, 0\} \mid i \in I \big\}$,
 we can narrow the possibilities for $\{a,b,c\}$
to:
$$
\{ 0,0,0\}, \{0,0,i \},  \{ 0,i,i\},\{i,i,i\}, \{i,i,j \},  
 $$
with $i$ and $j$ distinct indices in $I$. 
 In the first four cases the respective sets $P_{\{ 0,0,0\}}$, $P_{\{ 0,0,i\}}$, $P_{\{ 0,i,i\}}$ and $P_{\{ i,i,i\}}$ are contained in $T-X$. Therefore none of those possibilities could actually occur. The only possibility left is $\{a, b, c\}= \{i,i,j \}$, and then 
$P_{\{i,i,j\}} \subseteq T$.

\noindent   (ii) $\Rightarrow$ (iii) is clear, since $\{i, i\}, \{0, j\} \in P_{\{i, i, j\}}$.

\noindent (iii) $\Rightarrow$ (i). Let $i, j \in I$ be such that $i \neq j$ and $\{i, i\}, \{0, j\} \in T$. Using that $T$ is generalised nice, we get that $P_{\{i, i, j\}} \subseteq T$. 
Notice that both $\{i, i\}, \{0, j\}  \in T- X$. 
 If $T - X$ were a generalised nice set, then $P_{\{i, i, j\}} \subseteq T - X$. 
This is clearly untrue, since $\{i,   j\}  \in P_{\{i, i, j\}} $ but   $\{i,   j\}  \in X$.
\smallskip  

\noindent (ii) $\Rightarrow$ (iv). Let $i, j \in I$ be such that $i \neq j$ and $\langle P_{\{i, i, j\}} \rangle \subseteq T$. 
In particular, we have that $\{i, j\}, \{i\ast j, i\} \in T \cap X$. But $P_{\{i, j, i\}} \nsubseteq T \cap X$ 
($\{i, i\}\notin X $)
and (iv) holds.   
\smallskip  

\noindent (iv) $\Rightarrow$ (ii).  
Since we are assuming that $T \cap X$ is not generalised nice, there exist 
pairwise distinct $i,j,k\in I$ with 
$\{i,j\}, \{i \ast j, k\} \in T\cap X$
but $P_{\{i, j, k\}} \nsubseteq T \cap X$.  
If $i,j,k$ were generative, then  all the elements in $P_{\{i, j, k\}} $ would belong to $X$ and,
as $P_{\{i, j, k\}} \subseteq T  $,  we  would get the contradiction $P_{\{i, j, k\}} \subseteq T \cap X$.
This forces $k$ to be either $i$ or $j$  
($k\neq i\ast j$ because $\{i\ast j, k\}\in   X).$
 Relabelling the indices if necessary we can assume $k=i$, so
 that $P_{\{i, j, i\}} \subseteq T$.   
As  $T$ is generalised nice, we also find  $\langle P_{\{i, j, i\}} \rangle  \subseteq T$.
\end{proof}

\begin{cor}
Let $T$ be a generalised nice set such that $T \cap X \neq \emptyset$ and $T - X \neq \emptyset$. Then $T - X$ is generalised nice if and only if $T \cap X$ is so.
\end{cor}


\section{Generalised nice sets contained in $X$} \label{gnsXSection}

We begin by characterising the generalised nice sets that are contained in $X$. 

\begin{prop} \label{gnsX}
A subset $T$ of $X$ is a generalised nice set if and only if it satisfies the following two conditions:
\begin{enumerate}
\item[\rm (i)] $T$ is a nice set.
\item[\rm (ii)] There is no $i \in I$ such that $\{i, j\}, \{i, i \ast j\} \in T$ for $j \in I$ with $j \neq i$. 
\end{enumerate}
\end{prop}

\begin{remark}
Notice that condition (ii) above can be rephrased as follows:
\begin{center}
(ii)' $\quad |T \cap X_{\ell_{ij}}| \leq 1$, where $X_{\ell_{ij}} = \big\{\{i, j\}, \{i, i\ast j\}, \{j, i \ast j\} \big\}$.
\end{center}
\end{remark}

\begin{proof}
Suppose first that $T \subseteq X$ is a generalised nice set. Then $T$ clearly satisfies (i). Suppose  that $T$ does not satisfy (ii). Then there exists $i \in I$ such that $\{i, j\}, \{i, i \ast j\} \in T$ for some $j \in I$, $j \neq i$. From here we obtain that $P_{\{i, j, i\}} \subseteq T$, which implies that $\{i, i\} \in T$, a contradiction, since we are assuming that $T \subseteq X$. Conversely, assume that $T \subseteq X$ satisfies (i) and (ii). Take $i, j, k \in  I_0$ such that $\{i, j\}, \{i \ast j, k\} \in T$. Notice that necessarily $i,j,k\in I$ and $i \neq j$ since $T \subseteq X$, by hypothesis. Similarly, $k \neq i \ast j$. Moreover, $k \notin \{i, j\}$ by (ii). These considerations show that $i, j, k$ are generative, so $P_{\{i, j, k\}} \subseteq T$ by (i).
\end{proof}

\begin{cor} \label{somegns}
  There are 14 generalised nice sets contained in $X$ up to collineation: the empty set and 
\smallskip

- Cardinal 1: $\, \, \big\{ \{1, 2\} \big\}$.

\smallskip

- Cardinal 2: $\, \, \big\{ \{1, 2\}, \, \{1, 3\} \big\}, \, \,
\big\{ \{1, 2\}, \, \{6, 7\} \big\}$.

\smallskip

- Cardinal 3: $\, \,
\big\{ \{1, 2\}, \, \{1, 3\}, \{1, 4\} \big\}, \, \,
\big\{ \{1, 2\}, \, \{1, 3\}, \{1, 7\} \big\}$,  

\smallskip 

$\qquad \qquad \qquad \big\{ \{1, 2\}, \, \{1, 6\}, \{2, 6\} \big\}, \, \, \, \big\{ \{1, 2\}, \, \{1, 6\}, \{6, 7\} \big\}$,
$\big\{ \{2, 5\}, \, \{3, 6\}, \{4, 7\} \big\}$.

\smallskip

- Cardinal 4: 
$\, \,
\big\{ \{1, 2\}, \, \{1, 6\}, \{1, 7\}, \{2, 6\} \big\}, \, \,
\big\{ \{1, 2\}, \, \{1, 6\}, \{2, 7\}, \{6, 7\} \big\}.$

\smallskip

- Cardinal 5: 
$\, \,
\big\{ \{1, 2\}, \, \{1, 6\}, \{1, 7\}, \{2, 6\}, \{2, 7\} \big\}$.

\smallskip

- Cardinal 6: $\, \, 
\big\{ \{3, 4\}, \, \{3, 6\}, \{3, 7\}, \{4, 6\}, \{4, 7\}, \{6, 7\} \big\}$,

$\qquad \qquad \qquad \, 
\big\{ \{1, 2\}, \, \{1, 3\}, \{2, 3\}, \{1, 7\}, \{2, 6\}, \{3, 5\} \big\}$.
\end{cor}

\begin{proof}
  The list $\{T_i:i=1\dots24\}$ in \cite[Theorem 3.27]{Paper1} exhibits all the nice sets up to collineations. (Two subsets of $X$ are collinear   if and only if they are so as subsets of $X_0$.) 
Now we have only to check which of these sets satisfy the condition (ii) in Proposition~\ref{gnsX}. 
Simple inspection of the elements tells us that   $T_{2}$, $T_{3}$, $T_{5}$, $T_{7}$, $T_{8}$,
$T_{10}$, $T_{11}$, $T_{12}$, $T_{15}$, $T_{16}$, $T_{18}$, $T_{19}=X_{\ell_{12}^c}$, and $T_{21}=P_{\{1,2,3\}}$ 
are the only non-empty nice sets which are also generalised nice sets. 
 \end{proof}


\section{Generalised nice sets having an empty intersection with $X$} \label{TintXempty}


In this section, we determine all the generalised nice sets (up to collineations) that have an empty intersection with $X$. 
We begin by noticing   a trivial, but useful, fact.

\begin{lemma} \label{erratillas}
If $T$ is a generalised nice set such that $T\cap X = \emptyset$, there is no $i,j\in I$ distinct such that $\{0,i \} ,\{j,j \} \in T$. 
\end{lemma}

\begin{proof}
This is clear since $\{i,j\} \in P_{\{j,j,i\}}\cap X.$ (Alternatively, it is an immediate consequence of (iii)$\Rightarrow$(iv) in Proposition~\ref{charTintXnotgns}, since $\emptyset$ is a generalised nice set.)
\end{proof}

\begin{prop} \label{way} 
If $T$ is a  generalised nice set such that $T\cap X = \emptyset$, then either
\begin{itemize}
\item[\rm(a)] there is $i\in I$ with $T=P_{\{ 0,i,i\}}=  \big\{ \{ 0,i\},\{i,i \},\{ 0,0\} \big\}$;
\item[\rm(b)] or $\big\{  \{ 0,0\} \big\}\subseteq T\subseteq X_F=\big\{  \{ 0,i\} :i\in I_0\big\}$;
\item[\rm(c)] or  $  T\subseteq X_E-\big\{  \{ 0,0\} \big\}=\big\{  \{ i,i\} :i\in I\big\}$.
\end{itemize}
Moreover, any set $T \subseteq  X_0-X$ satisfying any of the situations described in {\rm(a)}, {\rm(b)} or {\rm(c)} is a generalised nice set.
\end{prop}

\begin{proof}
First, it is clear that all these sets are generalised nice. 
In case (b), for each $\{ 0,i\} \in T$,  we must have $P_{\{ 0,0,i\}}=  \{ \{ 0,i\},\{ 0,0\} \} \subseteq  T$.
 In case (c) there is nothing to check because there are no $a,b,c$ with $\{ a,b\} ,\{ a\ast b,c\} \in T$ since  $\{ 0,c\} \notin T$.
 
 Second, assume $\emptyset\ne T \subseteq  X_0-X$ is generalised nice.
 By Lemma~\ref{erratillas}, either  there is $i\in I$ with $\{ 0,i\},\{ i,i\} \in T$, or  $T\subseteq X_F=\big\{  \{ 0,i\} :i\in I_0\big\}$ or $T\subseteq X_E=\big\{  \{ i,i\} :i\in I_0\big\}$. In the first situation $\{ 0,0\} \in P_{\{ 0,i,i\}}  \subseteq  T$, and then $P_{\{ 0,i,i\}}  = T$ because any  element in $T-P_{\{ 0,i,i\}}$ would be either $\{ 0,j\}$ with $j\ne i$ (a contradiction with $\{ i,i\}\in T$) or  $\{ j,j\}$ with $j\ne i$ (a contradiction with $\{ i,0\}\in T$).
 If $T\subseteq X_F$, 
  either $ \{ 0,0\} \in T$ or there is $i\in I$ with $ \{ 0,i\}=\{ i,0\} \in T$. 
  Either way, $ \{ 0,0\} \in  T$, since $\{ 0,0\} \in P_{\{ 0,i,0\}} $.
 If $T\subseteq X_E$,  note   that $ \{i,i\} ,\{0,0\} \in T$ would imply $ \{ 0,i\} \in P_{\{ i,i,0\}}\subseteq T$, another contradiction. So, if
 $\big\{  \{ 0,0\} \big\}\ne T\subseteq \big\{  \{ i,i\} :i\in I_0\big\}$,   
 we can rest assured  that $ \{ 0,0\} \notin T$.
\end{proof}

 The above proposition  describes all the generalised nice sets; now we need only be cautious to avoid collinear repetitions.
  First, $P_{\{ 0,i,i\}}=\big\{\{0,0\}, \{0,i\}, \{i, i\}\big\}\sim_c P_{\{ 0,1,1\}}$ for any $i\in I$. Second, if $J,J'\subseteq I$, 
  then the generalised nice sets
  $\big\{  \{ 0,0\} ,\{ 0,i\} :i\in J\big\}$ and  $\big\{  \{ 0,0\} ,\{ 0,i\} :i\in J'\big\}$, are collinear if, and only if, $J$ and $J'$ are so. This means that we have 
  one possibility for each cardinality $|J|=0,1,2,5,6$ or $7$ (for instance, choosing 5 indices from $I$ is equivalent to choosing the remaining two). The other two possibilities 	
  are  $|J|=3 $ or $4$. 
  If $|J|=3$, we must distinguish between whether the three elements in $J$ form a line or not. Similarly, if $|J|=4$ the complementary set, $I-J,$ could be a line or not.
   This leaves 10  different possibilities wherein $\big\{  \{ 0,0\} \big\}\subseteq T\subseteq X_F$.
   Finally the discussion for $\big\{   \{ i,i\} :i\in J\big\}$ with $J$ a subset of $I$, is exactly the same. 
   That set is collinear to $\big\{   \{ i,i\} :i\in J'\big\}$ if, and only if, $J$ and $J'$ are collinear. 
   Hence, the same possibilities arise as before, that is, 10 with $T\subseteq X_E$, including the empty set. 
   In total, we have 21   generalised nice sets up to collineations, which do not intersect $X$. These are 
  compiled in the table below.
  \smallskip

 \begin{longtable} [c] { | c | c | }
    \hline
    \multicolumn{2} { | c | }{{\bf Generalised nice sets $T$ such that $T \cap X = \emptyset$}} \\
    \hline
    $|T|$ & All possible $T$s \\
    \hline
    \endfirsthead
    \hline
    \hline
    \endhead
    \hline
    \endfoot
    \hline
    \endlastfoot
    0 & $\emptyset$ \\
    \hline
    1 & $\big\{ \{0, 0\}\big\}$ \\
    & $\big\{ \{1, 1\}\big\}$ \\
    \hline 
    2 & $\big\{ \{0, 0\}, \{0, 1\} \big\}$ \\
      & $\big\{ \{1, 1\}, \{2, 2\} \big\}$  \\
    \hline
    3 & $\big\{ \{0, 0\}, \{0, 1\}, \{0, 2\}\big\}$ \\ 
      & $\big\{ \{0, 0\}, \{0, 1\}, \{1, 1\} \big\}$ \\
      & $\big\{ \{1, 1\}, \{2, 2\}, \{3, 3\}\big\}$ \\ 
      & $\big\{ \{1, 1\}, \{2, 2\}, \{5, 5\}\big\}$ \\
      \hline 
    4 & $\big\{ \{0, 0\}, \{0, 1\}, \{0, 2\}, \{0, 3\}\big\}$ \\ 
      & $\big\{ \{0, 0\}, \{0, 1\}, \{0, 2\}, \{0, 5\}\big\}$ \\
      & $\big\{ \{1, 1\}, \{2, 2\}, \{3, 3\}, \{4, 4\} \big\}$\\
      & $\big\{ \{1, 1\}, \{2, 2\}, \{3, 3\}, \{5, 5\} \big\}$ \\
      \hline 
    5 & $\big\{ \{0, 0\}, \{0, 1\}, \{0, 2\}, \{0, 3\}, \{0, 4\}\big\}$\\
      & $\big\{ \{0, 0\}, \{0, 1\}, \{0, 2\}, \{0, 3\}, \{0, 5\}\big\}$ \\
      & $\big\{ \{1, 1\}, \{2, 2\}, \{3, 3\}, \{4, 4\}, \{5, 5\} \big\}$ \\
    \hline
    6 & $\big\{ \{0, 0\}, \{0, 1\}, \{0, 2\}, \{0, 3\}, \{0, 4\}, \{0, 5\}\big\}$ \\ 
      & $\big\{ \{1, 1\}, \{2, 2\}, \{3, 3\}, \{4, 4\}, \{5, 5\}, \{6, 6\} \big\}$ \\  
    \hline 
    7 & $\big\{ \{0, i\} \mid i = 0, \ldots, 6\big \}$ \\ 
      & $\big\{ \{i, i\} \mid i = 1, \ldots, 7 \big\}$ \\  
  \hline
    8 & $X_F=\big\{ \{0, i\} \mid i = 0, \ldots, 7\big \}$ \\ 
    \hline
    \caption{\label{gnsEmptyX} Generalised nice sets $T$ with $T \cap X = \emptyset$}
      \end{longtable}


\section{Generalised nice sets whose intersection with $X$\\ is not generalised   nice.}\label{partI}


Here, we focus our attention on the generalised nice sets $T$ such that  $ T \cap X$ is not generalised nice
 (in particular $\emptyset \neq T\cap X$  and $T\not\subseteq X$). Proposition \ref{charTintXnotgns} tells us that this is equivalent to say that $\langle P_{\{i, j, i\}} \rangle \subseteq T$ for some distinct $i, j \in I$. Without loss of generality, we can assume that $i = 1$ and $j = 2$, 
that is,
$$
\langle P_{\{1,2,1\}} \rangle = \big\{\{0, 0\}, \{0, 1\}, \{0, 2\}, \{0, 5\}, \{1,1\}, \{1, 2\}, \{1, 5\} \big\} \subseteq  T.
$$ 
As one realises later, it is quite useful to investigate what happens when we add elements from $X_0$ to $T$. 

\begin{lemma} \label{TintXiijj}
Let $T$ be a generalised nice set   such that  $\langle P_{\{1, 2, 1\}} \rangle \subseteq  T$. 
If there exists   
 $\{i, j\} \in X^{(1)}=\big\{\{2, 5\}, \{3, 6\}, \{4, 7\}  \big\}$ such that $T \cap \big\{\{i, i\}, \{j, j\}, \{i, j\}\big\} \neq \emptyset$, 
 then $\big\{\{i, i\}, \{j, j\}, \{i, j\}\big\} \subseteq T$.
\end{lemma}

\begin{proof}
Suppose first that $\{i, i\} \in T$. Then from $\{0, 1\} \in T$ we obtain that $P_{\{i, i, 1\}} \subseteq T$
  and so $\{i, i \ast 1\} = \{i, j\} \in T$. Assume now that $\{i, j\} \in T$, then from $\{1, 1\} \in T$ we get that $P_{\{i, j, 1\}} \subseteq T$, and so $\{i, i\}, \{j, j\} \in T$.
\end{proof}

\begin{lemma} \label{TintXLemma00}
Let $T$ be a generalised nice set such that $\langle P_{\{1, 2, 1\}} \rangle \cup \big\{\{1,k\} \big\} \subseteq T$ 
  for some $k=3,4,6,7$. Then,
\begin{enumerate}
\item[\rm (i)]  $|T \cap X| \geq 4$ holds. More precisely, $\big\{\{1, 2\}, \{1, k\}, \{1, 5\}, \{1, k*1\} \big\} \subseteq T$.
\item[\rm (ii)] If $\{2, 5\} \in T$, then $\{2, k\} \in T$.
\end{enumerate}
\end{lemma}

\begin{proof}
 From $\{0, 1\} \in T$, we get that $P_{\{0, 1, k\}} \subseteq T$; in particular $\{0, k*1\} \in T$. 
 From $\{1, 1\}, \{0, k*1\}$ we get that $P_{\{1, 1, k*1\}} \subseteq T$ and so $\{1, k*1\} \in T$. 
 This proves (i).
Also, from $\{2, 5\} ,\{1, k\} \in T$     we derive that $P_{\{2, 5, k\}} \subseteq T$  and so $\{2, k\} \in T$.
\end{proof}

\begin{lemma} \label{TinTX0306}
Let $T$ be a generalised nice set such that $\langle P_{\{1, 2, 1\}} \rangle \subseteq T$. 
 If $T$ contains either $\{0, k\}$ or $\{1, k\}\ (k\ne1)$, then $\{\{0, k\}, \{0, k*1\}, \{1, k\}, \{1, k*1\}\} \subseteq T.$
\end{lemma} 

\begin{proof} We can assume $k=3,4,6,7$ (if $k=2,5$ there is nothing to prove).
Suppose first that $\{0, k\} \in T$. Then, from $\{1, 1\} \in T$ we get that $P_{\{1, 1, k\}} \subseteq T$ 
and so $\{1, k\}, \{1, k*1\} \in T$. Lastly, from $\{0, k\}, \{k, 1\} \in T$  we obtain that 
$P_{\{0, k, 1\}} \subseteq T$, which implies that $\{0, k*1\} \in T$. 
Second, assume  that   $\{1, k\} \in T$. We have that $\{1, k\}, \{1, k*1\} \in T$ by Lemma~\ref{TintXLemma00}(i). Then from $\{0, 1\}, \{1, k\} \in T$  we get that 
$P_{\{0, 1, k\}} \subseteq T$, and so  both $  \{0, k\} , \{0, k*1\}  \in T$. 
\end{proof}

\begin{prop} \label{TintXprop}
The following assertions hold for a generalised nice set $T$ such that $\langle P_{\{1, 2, 1\}} \rangle \subseteq T$. 
Take $j\in\{2,5\}$, $k\in\{3,4,6,7\}$.
\begin{enumerate}
\item[\rm (i)] If $\{j,k\} \in T$, then $T_{ (1, j,k) } \subseteq T \cap X$.
\item[\rm (ii)] If $\{j*k,k\} \in T$,  then $T_{ (1, j*k,k) } \subseteq T \cap X$. 
\end{enumerate}
\end{prop}

\begin{proof}
(i) Let us begin by recalling that  $\{1, j\},   \{1, j*1\}\in P_{\{1, 2, 1\}}\subseteq T$. 
From $\{1, j*1\}, \{j, k\} \in T$, we derive that $P_{\{1, j*1, k\}} \subseteq T$, and so $ \{k, j*1\}, \{j*1, k*1\} \in T$. 
Now, from $\{1, j\}, \{j*1, k*1\} \in T$ we get that $P_{\{1, j, k*1\}} \subseteq T$, and,
 from $\{1, j\}, \{j*1, k\} \in T,$ we obtain that $P_{\{1, j, k\}} \subseteq T$. 
 Altogether, we see $T_{ (1, j, k) } = P_{\{1, j, k\}} \cup P_{\{1, j, k*1\}} \cup P_{\{1, k, j*1\}}  \subseteq   T$.

\smallskip

\noindent (ii)  From $ \{j*k, k\},\{j,1\} \in T$, we derive that $P_{\{j*k, k, 1\}} \subseteq T$, and so both $  \{j*k, k*1\} \in T$ and $  \{j*k*1, k\} \in T$. 
With the elements $  \{j*k, k*1\},\{j*1,1\}  \in T$, we get  $P_{\{j*k, k*1, 1\}} \subseteq T$,  
and with $  \{j*k*1, k\} ,\{j*1,1\}\in T$, we obtain $P_{\{j*k*1, k, 1\}} \subseteq T$ too. 
These show that $T_{ (1, j*k,k) } \subseteq T \cap X$, as required. 
\end{proof}

\begin{cor} \label{TintX33}
Let $T$ be a generalised nice set such that $\langle P_{\{1, 2, 1\}} \rangle \subseteq T$. If $\{k,k\} \in T$, 
for $k\in\{3,4,6,7\}$, then $\{k, k*2\} \in T$. Moreover, $T \cap X$ is either $X$ or   $X - X_{\ell^C_{1k}}$. 
\end{cor}

\begin{proof}
From $\{k,k\}, \{0, 2\} \in T$ we derive that $P_{\{k, k, 2\}} \subseteq T$ and so $\{2, k\},\{k, k*2\}  \in T$. 
Now Proposition~\ref{TintXprop} tells us that $T_{ (1, 2, k) } \cup T_{ (1, k, k*2) } \subseteq T \cap X$. 
From here, $\{k, 2\},\{k*2, k*1\}  \in T$, which implies $P_{\{k, 2, k*1\}}\subseteq T$ and $\{k, k*1\}  \in T$.
The above argument can be repeated by replacing the index 2 with the index 5. 
Thus $X - X_{\ell^C_{1k}}=T_{ (1, 2, k) }\cup T_{ (1, 5, k) }\cup T_{ (1, k, k*2) } \cup T_{ (1, k, k*5) } \cup \big\{\{k, k*1\} \big\} \subseteq T \cap X$.
If this containment is strict, 
  \cite[Theorem 3.27]{Paper1} says that $T \cap X=X$, the only nice set with more than 15 elements.
  \end{proof}

\begin{theorem} \label{TintXthm}
The following hold for a generalised nice set $T$ such that $\langle P_{\{1, 2, 1\}} \rangle \subseteq T$:
\begin{enumerate}
\item[\rm (i)] If $|T \cap X_E| = 2$, then $|T \cap X| \in \{2, 4, 10, 15\}$. More precisely:
\begin{enumerate}
\item[-] if $|T \cap X| = 2$, then $T \cap X = \big\{\{1, 2\}, \{1, 5\} \big\};$
\item[-] if $|T \cap X| = 4$, then $T \cap X=\big\{\{1, 2\}, \{1, k\}, \{1, 5\}, \{1, k*1\} \big\}$ for some   $k=3,4,6,7$;
\item[-] if $|T \cap X| = 6$, then $T \cap X=X_{(1)}=\big\{\{1, 2\}, \{1, 3\}, \{1, 4\}, \{1, 5\}, \{1, 6\}, \{1, 7\} \big\};$
\item[-] if $|T \cap X| = 10$, then $T \cap X$ equals either $T_{(1,j,k)}$ or $T_{(1,j*k,k)}$ for some $j=2,5$ and $k=3,4,6,7.$
\end{enumerate}
\item[\rm (ii)] If $|T \cap X_E| > 2$, then $|T \cap X| \in \{3, 15, 21\}$. More concretely:
\begin{enumerate}
\item[-] if $|T \cap X| = 3$, then $T \cap X = X_{\ell_{12}}$; 
\item[-] if $|T \cap X| = 15$, then 
there is $ s\in I-\{1\}$ such that  $T \cap X=X - X_{\ell^C_{1s}}$;
\item[-] if $|T \cap X| = 21$, then $T \cap X = X$.
\end{enumerate}
\end{enumerate}
\end{theorem}

\begin{proof}
Let us begin by recalling that the possible  cardinals for non-empty nice sets are $ \{1, 2, 3, 4, 5, 6, 10, 15, 21\}$ by \cite[Theorem 3.27]{Paper1}. By Lemma~\ref{trivialfacts}(i),
$T\cap X$ is nice. 

\smallskip

\noindent (i) If $|T \cap X_E| = 2$, then $T \cap X_E = \big\{\{0, 0\}, \{1, 1\}\big\}$ since 
$\langle P_{\{1, 2, 1\}} \rangle \subseteq T$. 
In this case $\big\{\{2, 5\},\{3, 6\},\{4, 7\}\big\} \cap T=\emptyset$, taking into consideration  Lemma~\ref{TintXiijj}. 
This tells us that $T \cap X\ne X$. 

 First, let us check that the only possibilities for $|T \cap X| \le 9$ are $2$, $4,$ and $6$. 
We know   $\{1, 2\}, \{1, 5\} \in T$, so that $T \cap X$ has at least 2 elements.   
For any $j=2,5$, and any $k=3,4,6,7$, we know that $\{j,k\} \notin T$   because of Proposition~\ref{TintXprop}(i).
Similarly,   $\{j*k,k\} \notin T$ by Proposition~\ref{TintXprop}(ii).  Therefore, the only possible elements in $T \cap X$, different from $\{1, 2\}$ and $\{1, 5\}$, are   
$\{1, k\}\in T$ for some $k=3,4,6,7$. If $\{1, k\}\in T$, then   $\big\{\{1, 2\}, \{1, k\}, \{1, 5\}, \{1, k*1\} \big\} \subseteq T$
by Lemma~\ref{TintXLemma00}(i). 
If, moreover,  there is another element in $T \cap X$ besides those 4, it must   be $\{1, k'\}$ ($k'\in\{3,4,6,7\}$),  so that  again 
Lemma~\ref{TintXLemma00}(i) applies to get $X_{(1)} \subseteq  T \cap X$. 
These sets are necessarily equal, since there are no more elements in $X$ to add. 

Second, consider the case with $|T \cap X| =10$, that is, there   exists $\{i,j,k \}$ generative
with $T \cap X=T_{(i,j,k)}$. The fact that 
$\{1, 2\},\{1, 5\}\in T_{(i,j,k)}$  forces $i=1$. The possibilities for $j,k$ (interchangeable) with $j*k\ne1$   follow trivially. 

Third, if $|T \cap X| =15$, then $T \cap X=X - X_{\ell^C_{ij}}$ for some $i,j\in I$,  $i\ne j$. 
This set contains the elements $\{i,l\}, \{j,l\}$ and $\{i*j,l\}$, for any $l\in I$. In particular,
 $\{i,j\},\{i,i*j\},\{j,i*j\}\in T$, so that $P_{\{j, i, j\}} \subseteq  T$, which yields $\{j,j\}\in T\cap X_E$, and analogously $\{i,i\},\{i*j,i*j\}\in T\cap X_E$. This contradicts the fact   $|T\cap X_E|=2$.
\smallskip

\noindent (ii) Suppose now that $|T \cap X_E| > 2$. We distinguish two cases:

- Case 1: $T \cap X_E$  contains $\{k, k\}$ for some $k=3,4,6,7$. Then   Corollary~\ref{TintX33} applies to get that $T \cap X$  equals either $X$ or   $X - X_{\ell^C_{1k}}$.  
\smallskip

- Case 2: $T \cap X_E$ does not contain any element $\{k, k\}$ for $k=3,4,6,7$. So it contains either $\{2, 2\}$ or $\{5, 5\}$. Moreover,
it contains both of them by Lemma~\ref{TintXiijj}, and besides $\{2, 5\} \in T \cap X$.
If $|T \cap X| = 3$, then $|T \cap X| = \big\{\{1, 2\}, \{1, 5\}, \{2, 5\}\big\} = X_{\ell_{12}}$. Otherwise, let  
$|T \cap X| > 3$. Then there is $j=2,5$ and $k=3,4,6,7$ such that either $\{j, k\}$ or $\{j*k, k\}$ or $\{1, k\}$ or $\{1*k, k\}$ belongs to $T$ (all the  elements in $X-X_{\ell_{12}}$ take one of these forms). We claim that in any of these cases, $|T \cap X| \ge 11$. Indeed,
\begin{enumerate}
\item[(a)] For $\{j, k\}\in T $, Proposition~\ref{TintXprop}(i) tells us that $T_{ (1, j,k) } \subseteq T \cap X$.
The result is clear since $\{2, 5\} \notin T_{ (1, j,k )}$.
\item[(b)] For $\{j*k, k\}\in T $, again Proposition~\ref{TintXprop}(ii) tells us that $T_{ (1, j*k,k )} \subseteq T \cap X$  and  again
 $\{2, 5\} \notin T_{ (1, j*k,k ) }$.
 \item[(c)] For $\{1, k\}\in T$,  we find   $\{2, k\} \in T$ by Lemma~\ref{TintXLemma00}(ii) and then we apply item (a).
 \item[(d)] For $\{k*1, k\}\in T$, we use $\{1, 2\}\in T$ to get $P_{\{k*1,k,2\}}\subseteq T$, so we again obtain $\{2, k\} \in T$.
\end{enumerate}
As $T\cap X\ne X$ (since $\langle X\rangle=X_0$ contains $X_E$), there are $i,s$   with $T \cap X=X - X_{\ell^C_{i,s}}$. 
As in the proof of item (i), the three elements $\{i,i\},\{s, s\},\{i*s,i*s\}\in T\cap X_E$, so that $\{i,s,i*s\}=\{1,2,5\}$ and $T \cap X=X - X_{\ell^C_{1,2}}$.
\end{proof}

 
\begin{cor} \label{TintXcor}
Let $T$ be a generalised nice set such that $\langle P_{\{1, 2, 1\}} \rangle \subseteq T$. Then there is $\sigma\in S_*(I_0)$ with 
$\tilde\sigma(\langle P_{\{1, 2, 1\}} \rangle)=\langle P_{\{1, 2, 1\}} \rangle $ such that 
$\tilde\sigma(T \cap X)$ equals one of the following nice sets: 
\begin{itemize}
\item[-] $\big\{\{1, 2\}, \{1, 5\}\big\};$
\item[-] $X_{\ell_{12}}=\big\{\{1, 2\}, \{1, 5\}, \{2, 5\}\big\};$
\item[-] $\big\{\{1, 2\}, \{1, 3\}, \{1, 5\}, \{1, 6\}\big\}$;
\item[-] $X_{(1)}=\{\{1, l\}:l=2,\dots,7\};$
\item[-] either $T_{ (1, 2, 3) } $ or $T_{ (1, 3,4) } $;
\item[-] either $X - X_{\ell^C_{12}} $ or $X - X_{\ell^C_{13}} $;
\item[-] $X.$
\end{itemize}
\end{cor}

Keeping this in mind,  at last we can find the generalised nice sets whose intersection with $X$ is not generalised nice.

\begin{theorem} \label{GnsListThm}
There are 8 generalised nice sets $T$ up to collineation such that $T\cap X$ is not a generalised nice set:  
\smallskip

 - Cardinal 7: $\langle P_{\{1, 2, 1\}} \rangle= \big\{ \{0,0\},\{0,l\},\{1,l\}:l=1,2,5\big\}$;

 - Cardinal 10: $\langle P_{\{1, 2, 1\}} \rangle \cup \big\{\{2, 2\}, \{2, 5\}, \{5, 5\} \big\}=\langle P_{\{1, 2, 5\}} \rangle;$ 

- Cardinal 11: $\langle P_{\{1, 2, 1\}} \rangle \cup \big\{\{0, 3\}, \{0, 6\}, \{1, 3\}, \{1, 6\} \big\};$

 - Cardinal 15: $ \big\{ \{0,0\},\{0,l\},\{1,l\}:l=1,\dots, 7\big\}$;
  
 - Cardinal 19: $\big\{ \{0,0\},\{0,l\},\{1,l\}:l=1,\dots, 7\big\}\cup\big\{\{2, 3\}, \{5, 3\}, \{2,6\}, \{5,6\}\big\}$;
 
- Cardinal 26:  $X_0 - \big\{\{k, l\}:k,l\in\{3,4,6,7\}\big\};$

 - Cardinal 36: $X_0.$

\end{theorem}

\begin{proof}
We can assume, as in Theorem~\ref{TintXthm}, that $\langle P_{\{1, 2, 1\}} \rangle \subseteq T$. Then
 the set  $P:= \big \{\{0, 0\}, \{0, 1\}, \{0, 2\},  \{0, 5\},   \{1, 1\} \big\}$ is   contained in $T - X$. 
Recall that $T=  (T \cap  X)\cup (T \cap  X_E)\cup (T \cap  X_F)$.
In what follows, we examine all the possibilities for $T \cap X$ as per in Corollary~\ref{TintXcor}. 
We will repeatedly use  Theorem~\ref{TintXthm} (and its proof) without further mentioning it.

- If $T \cap X = \big\{\{1, 2\}, \{1, 5\}\big\}$, then $|T \cap  X_E|=2$. Also, $\{0, k\}\notin T$ if $k=3,4,6,7$ by Lemma~\ref{TinTX0306}, so that
$T = \langle P_{\{1, 2, 1\}} \rangle$.

-  If $T \cap X =  X_{\ell_{12}}$, then $\{2, 5\} \in T$, which implies $\{2, 2\}, \{5, 5\} \in T$ by Lemma~\ref{TintXiijj}. 
Again, for any $k=3,4,6,7$, $\{0, k\}\notin T$ by Lemma~\ref{TinTX0306} and $\{k, k\}\notin T$  by Corollary~\ref{TintX33}.
This implies $T - X = P \cup \big\{ \{2, 2\}, \{5, 5\} \big\}$, and so $T = \langle P_{\{1, 2, 1\}} \rangle \cup \big\{\{2, 2\}, \{2, 5\}, \{5, 5\} \big\}$.

- If $T \cap X = \big\{\{1, 2\}, \{1, 3\}, \{1, 5\}, \{1, 6\}\big\}$, we know $T \cap  X_E=\big\{\{0, 0\}, \{1,1 \}\big\}$ (by Theorem~\ref{TintXthm}) 
and $\{0, 3\}, \{0, 6\} \in T - X$, $\{0, 4\}, \{0, 7\} \notin T$, by Lemma~\ref{TinTX0306}. 
From here, $T - X = P \cup \big\{\{0, 3\}, \{0, 6\}\big\}$ and $T = \langle P_{\{1, 2, 1\}} \rangle \cup \big\{\{0, 3\}, \{0, 6\}, \{1, 3\}, \{1, 6\}\big\}$. 

- Similarly, if $T \cap X = X_{(1)}$, again $T \cap  X_E=\big\{\{0, 0\}, \{1,1 \}\big\}$ but now $\{0, k\} \in T - X$ for any $k=3,4,6,7$, by Lemma~\ref{TinTX0306}. 
We get $T - X = P \cup \big\{\{0, 3\}, \{0, 4\}, \{0, 6\}$, $\{0, 7\}\big\}$, and so 
$T = X_{(1)}\cup (T-X)= \langle P_{\{1, 2, 1\}} \rangle \cup \big\{\{0, 3\}, \{0, 4\}, \{0, 6\}, \{0, 7\}, \{1, 3\}, \{1, 4\}, \{1, 6\}, \{1, 7\} \big\}$. 

- Next, consider the case $|T \cap X |=10$, corresponding to  $T \cap  X_E=\big\{\{0, 0\}, \{1,1 \}\big\}$.
Considering a convenient collineation, 
we can assume without loss of generality that either $T \cap X = T_{ (1,2,3) }$ or $T \cap X = T_{ (1,3,4) }$.
In any case, 
since $\{1, k\} \in T$ for any $k=3,4,6,7$, then $\{0, k\}\in T$ by Lemma~\ref{TinTX0306}
and so $T-X=X_F\cup \big\{ \{1, 1\}\big\} $. 
We obtain $T = \langle P_{\{1, 2, 1\}} \rangle \cup X_F\cup (T\cap X)$. 
These sets coincide, respectively, with 
\vspace{-2pt}\begin{equation}\label{conjuntouno}
\big\{ \{0,0\},\{0,l\},\{1,l\}:l=1,\dots, 7\big\}\cup\big\{\{2, 3\}, \{5, 3\}, \{2,6\}, \{5,6\}\big\}
\end{equation}
and
\vspace{-2pt}\begin{equation}\label{conjuntodos}
\big\{ \{0,0\},\{0,l\},\{1,l\}:l=1,\dots, 7\big\}\cup\big\{\{3, 4\}, \{4,6\}, \{6,7\}, \{7,3\}\big\}.
\end{equation}
 In spite of the fact that there is no collineation $\sigma$ with $\tilde\sigma(\langle P_{\{1, 2, 1\}} \rangle)=\langle P_{\{1, 2, 1\}} \rangle $ and $\tilde\sigma(T_{ (1,2,3 )})=T_{ (1,3,4) }$, 
the sets in \eqref{conjuntouno} and \eqref{conjuntodos} are clearly collinear via, for instance, $\sigma_{143}$.

- Assume $T \cap X = X - X_{\ell^C_{1k}} $ for some $k\ne 1$. As $\{1, l\}\in T$ for all $l\in I$, Lemma~\ref{TinTX0306} says that $\{0, l\}\in T$ too.
We need to find the elements in  $T \cap X_E$.
We claim that $\{l, l\}\in T$ if $l=0,1,k,k*1$, and $\{l, l\}\notin T$ otherwise.
Indeed, as $\{k, 1\},\{k* 1,k\}\in T\cap X$, then $P_{\{k,1,k\}}\subseteq T$ and $\{k, k\}\in T$. 
Similarly, $\{k*1, k\ast 1\}\in T$ (as above, since $\{k*1, 1\},\{k* 1,k\}\in T\cap X$).
  If we further assume there is $l\in I$, $l\ne 1,k,k*1$  with $\{l, l\}\in T$, we will reach a contradiction. In fact, since $\{0, 1\}\in T$ we have $P_{\{l,l,1\}}\subseteq T$, so that 
$\{l, l*1\}\in T\cap X$. This forces either $l$ or $l*1$ to belong to $\ell_{1k}=\{1,k,k*1\}$, which is our contradiction. 
We conclude that the only elements in $X_0-T$ are those of the form $\{l, l'\} $ with $l,l'\ne 1,k,k*1$ ($l'$ possibly equal to $l$). 
It is clear that this set is collinear to the one provided by a different choice of $k$.

-  If $T \cap X = X$, then $\{2, 5\},  \{3, 6\}, \{4, 7\} \in T$, which implies that $X_E\subseteq T $ by Lemma~\ref{TintXiijj}.
Also $\{1, k\} \in T$ for any $k=3,4,6,7$, which implies  $\{0, k\}\in T$ by Lemma~\ref{TinTX0306}. 
Hence $T = X_0$.\smallskip

To finish, we have to prove that all the  sets in the statement are generalised nice. Since $T\cap X$    is nice and $\{0, 0\}\in T$, it is enough to check  the following conditions: 
\begin{itemize}
\item[$-$] If $  \{i, i\} \in T$, then $  \{0, i\}  \in T$;
\item[$-$] If $  \{0, i\}, \{i, j\} \in T$, then $  \{0, j\}, \{0, i* j\} \in T$;
\item[$-$] If $  \{0, i\}, \{j, j\} \in T$, then $  \{i, j\}, \{j, i* j\} \in T$;
\item[$-$] If $  \{i, i*j\}, \{i, j\} \in T$, then $  \{i, i\}, \{0, j\} \in T$;
\item[$-$] If $  \{i*j, i*j\}, \{i, j\} \in T$, then $  \{i, i\}, \{j, j\}, \{j, i* j\} ,\{i, i* j\} \in T$;
\end{itemize}
for any distinct indices $i,j\in I$. A direct verification is enough    to convince us   that all our sets fulfill these conditions.
%
%
%
%
\end{proof}



\section{Generalised nice sets whose   intersection with $X$ is   generalised.}\label{partII} 


Given $T \subseteq X_0$ generalised nice, we write $T$ as before $T = (T \cap X) \cup (T - X)$. In this section, we study the last case, namely, having both $T \cap X$ and $T - X$ non-empty and generalised nice (see Proposition~\ref{charTintXnotgns}). Notice that $T \cap X$ is a nice set by Lemma~\ref{trivialfacts}(i), and so (up to collineations) it is one of the nice sets listed in 
Corollary~\ref{somegns}. 
Regarding $T - X,$ since it clearly has empty intersection with $X$, it must be collinear to one of the generalised nice sets displayed in Table \ref{gnsEmptyX}. At first, one might think that this is simply a matter of putting together the puzzle pieces. Unfortunately, it is not as easy as it sounds since not every possible 
combination produces a generalised nice set. Let us illustrate this situation with an example:  consider $\big\{\{1, 1\}, \{2, 2\}\big\}$
and $\big\{\{2, 5\}, \{3, 6\}, \{4, 7\}\big\}$. The union of these two sets is not generalised nice, since the resulting set $T$ does not contain the set $P_{\{2, 5, 1\}}$.
\smallskip

We begin by proving   necessary conditions for $T - X$ to be generalised nice:

\begin{lemma} \label{condition}
Let $T$ be a generalised nice set such that $T - X$ is generalised nice. Then there are no $i, j \in I$ distinct satisfying   $\{i, i\}, \{j, i \ast j\} \in T$.  
\end{lemma}

\begin{proof}
 Suppose, on the contrary, that $\{j, i \ast j\}, \{i, i\} \in I$ for some $i, j \in I$ with $i \neq j$. But then $P_{\{j, i \ast j, i\}} \subseteq T$. From here we get that $\{i, j\} \in T$ and using that $\{i \ast j, j\} \in T$ we derive that $P_{\{i, j, j\}} \subseteq T$. Then $\{0, i\}, \{j, j\} \in T$, and Proposition \ref{charTintXnotgns} yields that $T - X$ is not generalised nice, 
a contradiction. 
\end{proof}

\begin{lemma} \label{0ielements}
The following conditions hold for a generalised nice set $T$ containing $\{0, i\}$, for some $i \in I$. 
\begin{itemize}
\item[\rm (i)] If $\{i, j\} \in T \cap X$ for $j \in I$ with $j \neq i$, then $\{0, j\}, \{0, i \ast j\} \in T - X$.
\item[\rm (ii)] If $\{j, k\} \in (T \cap X) \cap X^{(i)}$   for $j, k \in I$, 
then  $\{0, j\}, \{0, k\} \in T - X$.
\end{itemize}
\end{lemma}

\begin{proof}
(i) Apply that $T$ is generalised nice to get that $P_{\{0, i, j\}} \subseteq T$, which implies that $\{0, j\}, \{0, i \ast j\} \in T$.

\smallskip

\noindent (ii) Since $j \ast k = i$, from $\{j, k\}, \{i, 0\} \in T$, we have that $P_{\{j, k, 0\}}  \subseteq T$, since $T$ is generalised nice. Thus, $\{0, j\}, \{0, k\} \in T - X$. 
\end{proof}

The next step is to think   about sufficient conditions to obtain a generalised nice set $T$ as a combination of
two generalised nice sets 
$T \cap X$ and $T - X$. 
Recall from Proposition~\ref{way} (applied to $T-X$) that  either
$T-X\subseteq X_E-\big \{\{0, 0\}\big\}=:X_E^*$, or $\{0, 0\}\in T-X\subseteq X_F$ or $T-X=P_{\{0, i, i\}}$ for some $i\in I$, and any such set is generalised nice. We analyse these 3 situations in the following theorem.

\begin{theorem}\label{todo}
Take $\emptyset\ne S\subseteq X$ a generalised nice set, and $  J\subseteq I$.    
\begin{itemize}
\item[\rm (a)]  $S_J:=S\cup \big\{ \{j, j\}:j\in J\}$ is a generalised nice set if and only if $S\cap X^{(j)}=\emptyset$ for any $j\in J$. 

\noindent This is equivalent to the fact $J\cap J_S=\emptyset$, for $J_S:=\{a*b:\{a,b\}\in S\}$.
\item[\rm (b)]  $S'_J:=S\cup \big\{  \{0, 0\},\{0, j\}:j\in J\}$ is a generalised nice set 
 if, and only if, we have either $\ell_{ab}\cap J=\emptyset$, or $\ell_{ab}\subseteq 
J$, for any $\{a,b\}\in S$.

\noindent Denote by $I_S:=\cup_{\{a,b\}\in S}\ell_{ab}=\{a,b,a*b:\{a,b\}\in S\}$. Then
\begin{itemize}
\item[\rm (i)] If $J\cap I_S=\emptyset$, then $S'_J$ is a generalised nice set.  These are equivalent conditions if $|J|\le2$.
\item[\rm (ii)] If $I_S\subseteq J$, then $S'_J$ is a generalised nice set.  These are equivalent conditions if $|J|\ge5$.
\item[\rm (iii)] In case $I_S=I$,  the only generalised nice set  $T=S\cup (T-X)$ with $\emptyset\ne  T-X\subseteq X_F$ is that one with $T-X=X_F$.

\end{itemize}

\item[\rm (c)]  $\tilde S_i:=S\cup P_{\{0, i, i\}}$ is a generalised nice set if and only if $i\notin I_S$.

\end{itemize}
\end{theorem}

\begin{proof} 

\noindent (a) Let us consider first the case with $T-X\subseteq X_E^*$, that is, $T=S_J$.
 In Lemma~\ref{condition} we saw that in order for $S_J$ to be generalised nice, it is necessary that $S\cap X^{(j)}=\emptyset$ for any $j\in J$.
However, this condition is  also sufficient.
 Indeed, we have to check that $\{k,l\},\{k\ast l, m\}\in  S_J$ implies $P_{\{k,l,m\}}\subseteq S_J$.
If both  pairs are in $X$, then there is nothing to check, since $S=S_J\cap X$ is generalised  nice. 
We need only consider the case wherein one of these two pairs is $\{j,j\}$, with $j\in J$. If
$\{j,j\}=\{k,l\}$, then we have the contradiction $\{0,m\}\in S_J$.
Otherwise, $\{j,j\}=\{k\ast l, m\}$, so that $k\ast l= j$, that is, $\{k,l\}\in   X^{(j)}$.  However,
$\{k,l\}\in S $ (if $\{k,l\}\in S_J-S$, then $k\ast l= 0$, but $j\ne 0$), contradicting the hypothesis $S\cap X^{(j)}=\emptyset$. 
\smallskip

\noindent (b) Next assume $S'_J$ is a generalised nice set and take $\{a,b\}\in S$
such that $\ell_{ab}\cap J\ne\emptyset$. If either $a$ or $b$ belongs to $J$, then $\ell_{ab}=\{a,b,a*b\}\subseteq J$ by Lemma~\ref{0ielements}(i). 
 Otherwise, the index belonging to $J$ is $a*b$ and we
  get $\ell_{ab} \subseteq J$ by Lemma~\ref{0ielements}(ii). Conversely,  assume that 
for any $\{a,b\}\in S$  
either $\ell_{ab}\cap J=\emptyset$ or $\ell_{ab}\subseteq J$,  and let us prove that $S'_J$ is generalised  nice.
Take $\{k,l\},\{k\ast l, m\}\in S'_J$ and let us check that $P_{\{k,l,m\}}\subseteq  S'_J$.
If both pairs are in $X$ there is nothing to check, so assume one of them coincides with  either $\{0, 0\}$ or $\{0 ,  j\}$ for some $j\in J$.
If $k=l=m=0$, then trivially $P_{\{0,0, 0\}}=\big\{\{0 ,  0\}  \big\}\subseteq S'_J$.
If $k=l=0\ne m$ then $\{0 ,  m\}\in S'_J$,   and it is again evident that  $P_{\{0,0, m\}}=\big\{\{0 ,  0\},\{0 ,  m\}  \big\}\subseteq S'_J$.
It is not possible that $\{k,l\}\ne \{0 ,  0\} =\{k\ast l, m\}$, since $k=l$ but the only element in $S'_J\cap X_E$ is $\{0 ,  0\} $.
So assume that  neither of the  pairs is $\{0 ,  0\}$. 
If $\{k* l,m\}=\{0 ,  j\}$ with $j\in J$, then $k=l$, which is a contradiction since $\{k , k\}\notin S'_J$.
The only possibility left is $\{k ,  l\}=\{0 ,  j\}$, which gives $k*l=0*j=j$ and $\{j, m\}\in S$.
As besides $j\in\ell_{jm}\cap J $, the hypothesis gives $\ell_{jm}\subseteq J$ and so the three $\{0 ,  j\},\{0 ,  m\} $ and $\{0 ,  j*m\}$ belong to $S'_J$. 
Thus $P_{\{0,j,m\}}=\big\{ \{0 ,  j\},\{0 ,  m\} ,\{0 ,  j*m\},\{j, m\} \big\}\subseteq  S'_J$, which finishes case (b).

Next, let us consider the  particular cases mentioned within (b).

\noindent (i)   As for any $\{a,b\}\in S$ we have $\ell_{ab}\subseteq I_S$,   if $J\cap I_S=\emptyset$,  then 
 $\ell_{ab}\cap J=\emptyset$ and $S'_J$ is generalised   nice. On the other hand, if $|J|\le2$ and $S'_J$ is generalised  nice, then 
 for any $\{a,b\}\in S$, necessarily $\ell_{ab}\cap J=\emptyset$ ($J$ does not contain a line with 3 elements), which is to say $J\cap I_S=\emptyset$.
 
 \noindent (ii) If $I_S\subseteq J$, then, for any $\{a,b\}\in S$  we have $\ell_{ab}\subseteq J$, 
 which is one of the conditions that ensures that $S'_J$ is  generalised  nice. In case $|J|\ge5$, we can prove the converse: if $S'_J$ is generalised  nice, then 
  $\ell_{ab}\subseteq J$ for any $\{a,b\}\in S$, since  $J$ intersects any of the 7 lines ($I-J$ has at most   two elements, so it does no contain any line).

  \noindent (iii) First, it is clear that if $J=I$, then $S'_J$ is a  generalised nice set by  (ii). Second, let us check that we reach a contradiction if we assume 
 $I_S=I$, $  \emptyset\ne J\ne I$, and $S'_J$   a  generalised nice set. We can take $j\in J$ and 
 $a \in I-J$, by hypothesis on $J$.
 As $j\in I_S$, there is $k\in I$ with either $\{j, k\}\in S$ or $\{j*k, k\}\in S$. 
 As $j\in \ell_{jk}\cap J $, then $\ell_{jk} \subseteq J$.
 Also, $a\in I_S$ means that there is $b\in I$ with either $\{a, b\}\in S$ or $\{a*b, b\}\in S$. 
 As $\ell_{ab}\not\subseteq J$, this means $\ell_{ab}\cap J=\emptyset$. 
  Any two lines intersect, so that $\ell_{ab}\cap\ell_{jk}\ne \emptyset$, and any element $c\in \ell_{ab}\cap\ell_{jk}$ would be in $J$ and  
  {not in} $J$, a contradiction.
  \smallskip

\noindent (c) Finally, assume that $\tilde S_i=S\cup P_{\{0, i, i\}}$ is a generalised nice set. If $i\in I_S$, then there is $k\in I$
different from $i$ such that 
either $\{i, k\}\in S$ or $\{i*k, k\}\in S$. 
In the first case, $\{0,i\}\in \tilde S_i$ gives $P_{\{0,i,k \}}\subseteq \tilde S_i$ and $\{0,k\}\in \tilde S_i-X$. In the second case,  
$P_{\{i*k,k,0 \}}\subseteq \tilde S_i$ and we have the same contradiction $\{0,k\}\in \tilde S_i-X=P_{\{0, i, i\}}$. Conversely, take $i\notin I_S$   and 
$\{k,l\},\{k\ast l, m\}\in \tilde S_i$, with at least one of these two pairs not in  $S$. Let
us check that $P_{\{k,l,m\}}\subseteq  \tilde S_i$. If $\{k, l\}=\{0, 0\}$, then $k*l=0$ and the other element $ \{k\ast l, m\}\in \big\{ \{0, 0\},\{0, i\}  \big\}  $, but both $P_{\{0,0,0 \}}$ and $P_{\{0,0,i \}}\subseteq \tilde S_i$. If $\{k* l,m\}=\{0, 0\}\ne \{k, l\}$, then $k=l$, so $\{k, l\}=\{i,i\}$, and of course $P_{\{i,i,0 \}}\subseteq \tilde S_i$. Otherwise, either $\{k, l\}=\{i, i\}$  or $\{k, l\}=\{0, i\}$. In the first case   $\{k* l,m\}=\{0, i\}$ and $P_{\{i,i,i \}}\subseteq \tilde S_i$.
In the second case, $\{k* l,m\}=\{i,m\}\in S$. As $i\notin I_S$, then necessarily $m$ equals $0$ or $i$ and we are repeating   a previously discussed situation.
 \end{proof}


\subsection{Collecting the resulting generalised nice sets}
Let us continue by constructing generalised nice sets $T$ as combinations of the generalised nice sets $T - X$ and $T \cap X$, from Table \ref{gnsEmptyX} and Corollary~\ref{somegns} respectively (following the criteria to combine given in Theorem~\ref{todo}). We may assume that $T \cap X$ is, not only collinear but, equal to one of the generalised nice sets from Corollary~\ref{somegns} and go through all the possibilities for $T - X$ up to collineation. For instance if $T \cap X = \big\{\{1, 2\}\big\}$ and $T - X \sim_c \big\{\{1, 1\}\big\}$, 
then there is $i\in I$ with $T - X = \big\{\{i, i\}\big\}$, and Theorem~\ref{todo}(a) says that $T$ is a generalised nice set for all $i\ne 5$.
Among 
the six  possible combinations that give rise to a generalised nice set, there are only two not collinear: $\big\{\{1, 1\}, \{1, 2\}\big\}$ and $\big\{\{3, 3\}, \{1, 2\}\big\}$. 
Analogously,  Theorem~\ref{todo} permits us to know which elements from $X_0 - X$ can be added 
to   each of the cases displayed in Corollary~\ref{somegns} 
so that the resulting disjoint union is a generalised nice set; and our job consists of checking the possible collineations.

We have separate tables for each possible cardinality of $T - X$ and up to collinearity of $T-X;$ more precisely:
we will be looking at the following possibilities:
\begin{itemize}
 
\item[-]  $T-X\subseteq X_E-\big \{\{0, 0\}\big\}=X_E^*$;
\item[-]  $\{0, 0\}\in T-X\subseteq X_F$;
\item[-]   $T-X=P_{\{0, i, i\}}$ for some $i\in I$.
\end{itemize}

In order to apply Theorem~\ref{todo}, it is convenient to  first compute the sets $I_S$ and $J_S$ described therein, for any nice set $S$ in Corollary~\ref{somegns}. This computation is straightforward.
\begin{longtable} [c] { | c | c| c| }
    \hline
     $ \bf I_S$ & $\bf S$ & $\bf J_S$ \\
    \hline 
    \endfirsthead
    $\{1,2,5\}$  &  $\big\{ \{1, 2\}\big\}$            &  $\{5\}$ \\
    \hline
     $\{1,2,3,5,6\}$& $\big\{ \{1, 2\}, \{1, 3\}\big\}$   & $\{5,6\}$ \\
    \hline 
     $\{1,2,5,6,7\}$& $\big\{ \{1, 2\}, \{6, 7\}\big\}$   & $\{5\}$ \\
    \hline
    $I=\{1,2,3,4,5,6,7\}$ & $\big\{ \{2, 5\}, \{3, 6\}, \{4, 7\}\big\}$ & $\{1\}$   
  \\
   \hline
  $I$ & $\big\{ \{1, 2\}, \{1, 3\}, \{1, 4\}\big\}$ & $\{5,6,7\}$ 
  \\
  \hline
   $I$& $\big\{ \{1, 2\}, \{1, 3\}, \{1, 7\}\big\}$ & $\{4, 5, 6\}$ \\
 \hline
  $\{1,2,3,4,5,6\}=I-\{7\}$& $\big\{ \{1, 2\}, \{1, 6\}, \{2, 6\}\big\}$ & $\{3,4,5\}$ \\
  \hline
    $\{1,2,3,5,6,7\}=I-\{4\}$& $\big\{ \{1, 2\}, \{1, 6\}, \{6, 7\}\big\}$ & $\{3,5\}$
  \\ 
  \hline
   $I$& $\big\{ \{1, 2\}, \{1, 6\}, \{1, 7\}, \{2, 6\}\big\}$ & $\{3,4,5\}$ \\
  \hline
   $\{1,2,3,5,6,7\}=I-\{4\}$& $\big\{ \{1, 2\}, \{1, 6\}, \{2, 7\}, \{6, 7\}\big\}$ & $\{3,5\}$ \\
  \hline
  $I$& $\big\{ \{1, 2\}, \{1, 6\}, \{1, 7\}, \{2, 6\}, \{2, 7\}\big\}$ & $\{3,4,5\}$\\
  \hline
   $I$& $X_{\ell^c_{12}}$ & $\{1,2,5\}$ \\
    \hline
   $I$& $P_{\{1, 2, 3\}}$ & $\{4,5,6,7\}$\\
  \hline
    \caption{\label{IsyJs}  Auxiliary  indices for the Gns contained in $X$.  }
      \end{longtable}


\subsection{$T - X$ of cardinality 1}

  Recall  that the generalised nice sets contained in $X_0 - X$ of cardinality 1 are all of the form $\big \{\{i, i\} \big\}$, for $i \in I_0$. Moreover, Lemma~\ref{trivialfacts}(iv) tells us that $S \cup \big\{\{0, 0\}\big\}$  is generalised nice for any $S$ as in Corollary~\ref{somegns}. The resulting generalised nice sets have been collected in the table below.

\smallskip

\begin{longtable} [c] { | c | c| c| }
    \hline
    \multicolumn{3} { | c | }{{\bf Gns $T$ such that both $T \cap X$ and $T - X$ are non-empty gns}} \\
    \hline
     $T - X$ & $T \cap X$ & All possible $T$s \\
    \hline
    \endfirsthead
    \hline
    $T - X$ & $T \cap X$ & All possible $T$s \\
    \hline
    \endhead
    \hline
    \endfoot
    \hline
    \endlastfoot
    $\big\{ \{0, 0\}\big\}$  &  $\big\{ \{1, 2\}\big\}$            & $\big\{ \{0, 0\}, \{1, 2\}\big\}$ \\
    \hline
    $\big\{ \{0, 0\}\big\}$ & $\big\{ \{1, 2\}, \{1, 3\}\big\}$   & $\big\{ \{0, 0\}, \{1, 2\}, \{1, 3\}\big\}$ \\
    \hline 
    $\big\{ \{0, 0\}\big\}$ & $\big\{ \{1, 2\}, \{6, 7\}\big\}$   & $\big\{ \{0, 0\}, \{1, 2\}, \{6, 7\}\big\}$ \\
    \hline
    $\big\{\{0, 0\}\big\}$ & $\big\{ \{2, 5\}, \{3, 6\}, \{4, 7\}\big\}$ & $\big\{\{0, 0\}, \{2, 5\}, \{3, 6\}, \{4, 7\}\big\}$   
  \\
   \hline
  $\big\{\{0, 0\}\big\}$ & $\big\{ \{1, 2\}, \{1, 3\}, \{1, 4\}\big\}$ & $\big\{\{0, 0\}, \{1, 2\}, \{1, 3\}, \{1, 4\}\big\}$ 
  \\
  \hline
  $\big\{ \{0, 0\}\big\}$ & $\big\{ \{1, 2\}, \{1, 3\}, \{1, 7\}\big\}$ & $\big\{\{0, 0\}, \{1, 2\}, \{1, 3\}, \{1, 7\}\big\}$ \\
 \hline
 $\big\{ \{0, 0\}\big\}$ & $\big\{ \{1, 2\}, \{1, 6\}, \{2, 6\}\big\}$ & $\big\{\{0, 0\}, \{1, 2\}, \{1, 6\}, \{2, 6\}\big\}$ \\
  \hline
   $\big\{ \{0, 0\}\big\}$ & $\big\{ \{1, 2\}, \{1, 6\}, \{6, 7\}\big\}$ & $\big\{\{0, 0\}, \{1, 2\}, \{1, 6\}, \{6, 7\}\big\}$
  \\ 
  \hline
  $\big\{ \{0, 0\}\big\}$ & $\big\{ \{1, 2\}, \{1, 6\}, \{1, 7\}, \{2, 6\}\big\}$ & $\big\{\{0, 0\}, \{1, 2\}, \{1, 6\}, \{1, 7\}, \{2, 6\}\big\}$ \\
  \hline
  $\big\{ \{0, 0\}\big\}$ & $\big\{ \{1, 2\}, \{1, 6\}, \{2, 7\}, \{6, 7\}\big\}$ & $\big\{\{0, 0\}, \{1, 2\}, \{1, 6\}, \{2, 7\}, \{6, 7\}\big\}$ \\
  \hline
 $\big\{ \{0, 0\}\big\}$ & $\big\{ \{1, 2\}, \{1, 6\}, \{1, 7\}, \{2, 6\}, \{2, 7\}\big\}$ & $\big\{ \{0, 0\}, \{1, 2\}, \{1, 6\}, \{1, 7\}, \{2, 6\}, \{2, 7\}\big\}$\\
  \hline
  $\big\{ \{0, 0\}\big\}$ & $X_{\ell^c_{12}}$ & $\big\{\{0, 0\}\big\} \cup X_{\ell^c_{12}}$ \\
    \hline
  $\big\{ \{0, 0\}\big\}$ & $P_{\{1, 2, 3\}}$ & $\big\{\{0, 0\}\big\} \cup P_{\{1, 2, 3\}}$\\
  \hline
    \caption{\label{gnsLastcase00} $T - X = \big\{\{0, 0\}\big\}$.}
      \end{longtable}

\smallskip

Next, we consider generalised nice sets $T$ such that $T - X$ is of the form $\big\{\{i, i\}\big\}$ for $i \in I$. 
As a a direct application  of Theorem~\ref{todo}(a),  
with the notations therein, we obtain: 

\begin{cor}\label{le_conii}
If $S\subseteq X$ is a   generalised nice set, 
then $S_{\{ j\}}=S\cup \big\{\{j, j\}\big\}$ is a generalised nice set if and only if $S\cap X^{(j)}=\emptyset$, or equivalently, if
$j\notin J_S$.
\end{cor}

Thus, to obtain  the generalised nice sets $T$ such that $T - X$ is of the form $\big\{\{i, i\}\big\}$, 
 for $i \in I$ and $T\cap X$ coinciding with a fixed generalised nice set $S$,
we only have to be careful with how many equivalence classes appear. 
Take into account that  any  $\sigma$ such that $\tilde\sigma(S_{\{ i \}})= S_{\{ j \}}$  (for $i,j\notin J_S$) satisfies 
$\tilde\sigma(S)=S$ so that $ \sigma$ preserves $I_S$ and $J_S$, as well as the subsets of indices in $I_S-J_S$ with the same number of occurrences in elements of $S$.   Here we give a detailed description of how to use these tools.

\noindent   $-$   
 $S=\big\{ \{1, 2\}\big\}$. 
 As in the above corollary,   $S_{\{  i\}}=\big\{ \{1, 2\},\{i, i\}\big\}$ is a generalised nice set if and only if  $i\ne 5$. 
Now $S_{\{ 1 \}}$ is collinear to $S_{\{ 2 \}}$, for instance,
$\tilde\sigma(S_{\{ 1 \}})= S_{\{ 2 \}}$ for $\sigma=\sigma_{213}$.   
  Notice also that $S_{\{ 3 \}}\sim_c S_{\{  4\}}\sim_c S_{\{  6\}}\sim_c S_{\{ 7 \}}$ are collinear  to one another, since 
for any $k\notin\ell_{12}$ the collineation $\sigma=\sigma_{12k}$ satisfies
$\tilde\sigma(S_{\{ 3 \}})= S_{\{ k \}}$.
However, $S_{\{ 1 \}}$ is not collinear to $S_{\{ 3 \}}$ because we require that $\sigma$ be $I_S-J_S=\{1,2\}$-invariant,
 and so $\tilde\sigma (\{ 1,1 \})\ne \{ 3,3 \}$.

\noindent  $-$
$S=\big\{ \{1, 2\}, \{1, 3\}\big\}$.
Now $S_{\{  i\}} $ is  generalised nice set if    $i\ne 5,6$. Any collineation interchanging two of these sets has to
leave invariant $\{1\}$ and $\{2,3\}$, so either $\sigma=\id$ or $\sigma=\sigma_{132}$.
In   either case, $\sigma$ carries $7=2*3$ to 7 and hence
 $4=1*7$ to $4$. This means that $S_{\{ 2 \}}\sim_c S_{\{ 3 \}}$ ($\sigma_{132}$ is a suitable collineation) 
 but the $S_{\{ i \}}$'s for each of $i=1,2,4,7$ are pairwise not collinear.

\noindent  $-$
 $S=\big\{ \{1, 2\}, \{6, 7\}\big\}$. 
 While we can join $S$ with  $\big\{\{i, i\}\big\}$ for any $i \ne 5$ to find a generalised nice set, only two distinct equivalence
  classes appear: $S_{\{ 1 \}}\sim_c S_{\{ 2 \}}\sim_cS_{\{ 6 \}}\sim_c S_{\{ 7 \}}$
and $S_{\{ 3 \}}\sim_c S_{\{ 4 \}}$.
Note that  $S_{\{ 1 \}}\not\sim_c S_{\{ 3 \}}$, since any collineation would preserve the subsets $\{1,2,6,7\}$ and $\{3,4\}$.

\noindent  $-$
$S=\big\{ \{2, 5\}, \{3, 6\}, \{4, 7\}\big\}$ can be combined with $\big\{\{i, i\}\big\}$ for $i \ne 1$. 
There is only one class because all the elements $2,3,4,5,6,7$ play the same role in the set $S$. For instance $\sigma_{1ik}$  carries $S_{\{ 2 \}}$ to $S_{\{ i \}}$ for any $k\notin\ell_{1i}$.

\noindent  $-$
$S=\big\{ \{1, 2\}, \{1, 3\}, \{1,4\}\big\}$ can be enlarged with $\big\{\{i, i\}\big\}$ for $i = 1,2,3,4$ to get a generalised  nice set. 
Two classes arise: 
$S_{\{ 1 \}}$ and
$  S_{\{ 2 \}}\sim_cS_{\{ 3 \}}\sim_c S_{\{ 4\}}$, since $\sigma_{132}$ and $\sigma_{142}$ move $S_{\{ 2 \}}$ to $S_{\{ 3 \}}$ and $S_{\{ 4 \}},$ respectively. 

\noindent  $-$
$S=\big\{ \{1, 2\}, \{1, 3\}, \{1,7\}\big\}$. This case is very similar to the above, with two equivalence classes, $S_{\{ 1 \}}$ and
$  S_{\{ 2 \}}\sim_cS_{\{ 3 \}}\sim_c S_{\{ 7\}}$.

\noindent  $-$
$S=\big\{ \{1, 2\}, \{1, 6\}, \{6, 7\}\big\}$. 
Now  $S_{\{ i \}}$ is generalised nice whenever $i\notin J_S=\{ 3,5\}$. 
Any collineation 
between two of these sets  has to preserve       $\{1, 6\}$, $\{2,7\}$ (because of the number of occurrences in $S$), and $I-I_S=\{4\}$. 
So $S_{\{ 1 \}}$, $S_{\{ 2\}},$ and $S_{\{ 4 \}}$ cannot be collinear.
The map $\sigma_{673}$ allows us to see both that $S_{\{ 1 \}}$ is collinear to $S_{\{ 6 \}}$, and $S_{\{ 2 \}}$ to $S_{\{ 7 \}}$. This leaves
 just 3 
equivalence classes.

\noindent  $-$
$S=\big\{ \{1, 2\}, \{1, 6\}, \{1, 7\}, \{2, 6\}\big\}$. 
Notice first that the index $1$ appears   thrice in elements of $S$, the indices $2$ and $6$ each appear twice, and $7$ appears once. In this setting our collineations must preserve the sets $\{1\}$, $\{2,6\}$ and $\{7\}$, to give 3 classes: $S_{\{ 1 \}}$, $S_{\{ 2 \}}\sim_c S_{\{ 6\}}$,
 and $S_{\{7 \}}$.

\noindent  $-$
$S=\big\{ \{1, 2\}, \{1, 6\}, \{2, 7\}, \{6, 7\}\big\}$. As $I-I_S=\{4\}$, 
we see that $S_{\{ 4 \}}$ is not collinear to $S_{\{ i \}}$ for any $i\in I_S-J_S=\{1,2,6,7\}$. 
Furthermore, it is quite easy to check that $S_{\{ i \}}$ is collinear to $S_{\{ j\}}$ for any two indices $i,j\in I_S-J_S$.

\noindent  $-$
$S=\big\{ \{1, 2\}, \{1, 6\}, \{1, 7\}, \{2, 6\}, \{2, 7\}\big\}$. Working as in the above cases, we get two equivalence classes: $S_{\{ 1 \}}\sim_c S_{\{ 2 \}}$  and  $S_{\{ 6\}}\sim_c S_{\{ 7\}}$.

\noindent  $-$
$S=X_{\ell^c_{12}}$. Here there is only one possibility up to collineation, since $S_{\{ i \}}$ is generalised   nice only for $i\in I-\{1,2,5\}$. These four indices each play the same role in $S$ and it is simple to find the required collineation.

\noindent   $-$   
$S=P_{\{1, 2, 3\}}$. Again there is only one class $S_{\{ 1\}}\sim_c S_{\{ 2\}}\sim_c S_{\{ 3\}}$.

 We have obtained exactly 27 generalised nice sets $T$ such that $|T - X|=1$, not containing $\{0,0\}$. These are collected in the table below.

\begin{longtable} [c] { | c | c| c| }
    \hline
    \multicolumn{3} { | c | }{{\bf Gns $T$ such that both $T \cap X$ and $T - X$ are non-empty gns}} \\
    \hline
     $T - X$ & $T \cap X$ & All possible $T$s \\
    \hline
    \endfirsthead
    \hline
    \endhead
    \hline
    \endfoot
    \hline
    \endlastfoot
    $\big\{ \{1, 1\}\big\}$ &   $\big\{ \{1, 2\}\big\}$                                  & $\big\{ \{1, 1\}, \{1, 2\}\big\}$ \\
    $\big\{ \{3, 3\}\big\}$ &                                     & $\big\{ \{3, 3\}, \{1, 2\}\big\}$ \\
    \hline
    $\big\{ \{1, 1\}\big\}$ & $\big\{ \{1, 2\}, \{1, 3\}\big\}$   & $\big\{ \{1, 1\}, \{1, 2\}, \{1, 3\}\big\}$ \\
    $\big\{ \{2, 2\}\big\}$ &                                     & $\big\{ \{2, 2\}, \{1, 2\}, \{1, 3\}\big\}$ \\
    $\big\{ \{4, 4\}\big\}$ &                                     & $\big\{ \{4, 4\}, \{1, 2\}, \{1, 3\}\big\}$ \\
    $\big\{ \{7, 7\}\big\}$ &                                     & $\big\{ \{7, 7\}, \{1, 2\}, \{1, 3\}\big\}$ \\
   \hline 
    $\big\{ \{1, 1\}\big\}$ & $\big\{ \{1, 2\}, \{6, 7\}\big\}$   & $\big\{ \{1, 1\}, \{1, 2\}, \{6, 7\}\big\}$ \\
    $\big\{ \{3, 3\}\big\}$ &                                     & $\big\{ \{3, 3\}, \{1, 2\}, \{6, 7\}\big\}$ \\
    \hline
  $\big\{ \{2, 2\}\big\}$ &   $\big\{ \{2, 5\}, \{3, 6\}, \{4, 7\}\big\}$  & $\big\{ \{2, 2\}, \{2, 5\}, \{3, 6\}, \{4, 7\}\big\}$    
  \\
   \hline
  $\big\{\{1, 1\}\big\}$ & $\big\{ \{1, 2\}, \{1, 3\}, \{1, 4\}\big\}$   & $\big\{\{1, 1\}, \{1, 2\}, \{1, 3\}, \{1, 4\}\big\}$  
   \\ 
  $\big\{\{2, 2\}\big\}$ &                                           & $\big\{\{2, 2\}, \{1, 2\}, \{1, 3\}, \{1, 4\}\big\}$ \\ 
  \hline
 $\big\{ \{1, 1\}\big\}$ & $\big\{ \{1, 2\}, \{1, 3\}, \{1, 7\}\big\}$ & $\big\{\{1, 1\}, \{1, 2\}, \{1, 3\}, \{1, 7\}\big\}$ \\ 
  $\big\{ \{2, 2\}\big\}$ &  & $\big\{\{2, 2\}, \{1, 2\}, \{1, 3\}, \{1, 7\}\big\}$ \\ 
 \hline
  $\big\{ \{1, 1\}\big\}$ & $\big\{ \{1, 2\}, \{1, 6\}, \{2, 6\}\big\}$ & $\big\{\{1, 1\}, \{1, 2\}, \{1, 6\}, \{2, 6\}\big\}$ \\ 
  $\big\{ \{7, 7\}\big\}$ &  & $\big\{\{7, 7\}, \{1, 2\}, \{1, 6\}, \{2, 6\}\big\}$ \\
  \hline
  $\big\{ \{1, 1\}\big\}$ & $\big\{ \{1, 2\}, \{1, 6\}, \{6, 7\}\big\}$ & $\big\{\{1, 1\}, \{1, 2\}, \{1, 6\}, \{6, 7\}\big\}$ \\ 
  $\big\{ \{2, 2\}\big\}$ &  & $\big\{\{2, 2\}, \{1, 2\}, \{1, 6\}, \{6, 7\}\big\}$ 
  \\ 
  $\big\{ \{4, 4\}\big\}$ &  & $\big\{\{4, 4\}, \{1, 2\}, \{1, 6\}, \{6, 7\}\big\}$ \\
  \hline
  $\big\{ \{1, 1\}\big\}$ & $\big\{ \{1, 2\}, \{1, 6\}, \{1, 7\}, \{2, 6\}\big\}$ & $\big\{\{1, 1\}, \{1, 2\}, \{1, 6\}, \{1, 7\}, \{2, 6\}\big\}$ \\ 
  $\big\{ \{2, 2\}\big\}$ &  & $\big\{\{2, 2\}, \{1, 2\}, \{1, 6\}, \{1, 7\}, \{2, 6\}\big\}$ 
  \\ 
  $\big\{ \{7, 7\}\big\}$ &  & $\big\{\{7, 7\}, \{1, 2\}, \{1, 6\}, \{1, 7\}, \{2, 6\}\big\}$ \\
  \hline
$\big\{ \{1, 1\}\big\}$ & $\big\{ \{1, 2\}, \{1, 6\}, \{2, 7\}, \{6, 7\}\big\}$ & $\big\{\{1, 1\}, \{1, 2\}, \{1, 6\}, \{2, 7\}, \{6, 7\}\big\}$ \\ 
  $\big\{ \{4, 4\}\big\}$ &  & $\big\{\{4, 4\}, \{1, 2\}, \{1, 6\}, \{2, 7\}, \{6, 7\}\big\}$ 
  \\
  \hline
  $\big\{ \{1, 1\}\big\}$ & $\big\{ \{1, 2\}, \{1, 6\}, \{1, 7\}, \{2, 6\}, \{2, 7\}\big\}$ & $\big\{ \{1, 1\}, \{1, 2\}, \{1, 6\}, \{1, 7\}, \{2, 6\}, \{2, 7\}\big\}$ \\ 
  $\big\{ \{6, 6\}\big\}$ &  & $\big\{ \{6, 6\}, \{1, 2\}, \{1, 6\}, \{1, 7\}, \{2, 6\}, \{2, 7\}\big\}$
  \\
  \hline
  $\big\{ \{3, 3\}\big\}$ & $X_{\ell^c_{12}}$ & $\big\{\{3,3\}\big\} \cup X_{\ell^c_{12}}$ \\ 
  \hline
  $\big\{ \{1, 1\}\big\}$ & $P_{\{1, 2, 3\}}$ & $\big\{\{1, 1\}\big\} \cup P_{\{1, 2, 3\}}$ \\ 
  \hline
    \caption{\label{gnsLastcase} $T - X = \big\{\{i, i\}\big\}$ for some $i \in I$.}
      \end{longtable}

\subsection{$T - X$ of cardinality 2}
Let $T$ be a generalised nice set satisfying that $T - X$ is generalised nice 
with all its elements taking the form $\{i, i\}$ for $i \in I$. 
The result below is an immediate consequence of  Theorem~\ref{todo}(a).
 
\begin{cor}\label{le_conii2}
Let $S\subseteq X$ be a   generalised nice set, and $i,j\in I$, $i\ne j$.
Then $S_{\{i, j\}}$ is a generalised nice set if and only if  
$i,j\notin J_S$.  
\end{cor}

As we increase the cardinality of $T - X$, the size of the resulting $T$ (obtained by computing the union of the sets listed in the two first columns) obviously gets bigger. Due to this reason, we have chosen to leave the third column (as per in Table \ref{gnsLastcase}) out in the next few tables.  


\begin{longtable} [c] { | c | c|  }
    \hline
    \multicolumn{2} { | c | }{{\bf Gns $T$ such that both $T \cap X$ and $T - X$ are non-empty gns}} \\
   \hline
  $T - X$ & $T \cap X$ \\
    \hline
    \endfirsthead
    \hline
    $T - X$ & $T \cap X$ \\
    \hline
    \endhead
    \hline
    \endfoot
    \hline
    \endlastfoot
    $\big\{ \{1, 1\}, \{2, 2\} \big\}$ & $\big\{ \{1, 2\}\big\}$ \\
    $\big\{ \{1, 1\}, \{3, 3\} \big\}$ & \\
    $\big\{ \{3, 3\}, \{4, 4\} \big\}$ & \\
    $\big\{ \{3, 3\}, \{6,6\} \big\}$ & \\
    \hline
    $\big\{ \{1, 1\}, \{2, 2\} \big\}$ & $\big\{ \{1, 2\}, \{1, 3\}\big\}$ \\
    $\big\{ \{1, 1\}, \{4, 4\}\big\}$ &  \\
    $\big\{ \{1, 1\}, \{7, 7\}\big\}$ & \\
    $\big\{ \{2, 2\}, \{3, 3\} \big\}$ &  \\
    $\big\{ \{2, 2\}, \{4, 4\} \big\}$ &  \\
    $\big\{ \{2, 2\}, \{7, 7\} \big\}$ &  \\
    $\big\{ \{4, 4\}, \{7, 7\} \big\}$ &  \\
   \hline 
   $\big\{ \{1, 1\}, \{2, 2\} \big\}$ & $\big\{ \{1, 2\}, \{6, 7\}\big\}$ \\
    $\big\{ \{1, 1\}, \{3, 3\}\big\}$ &  \\
    $\big\{ \{1, 1\}, \{6, 6\}\big\}$ &  \\
$\big\{ \{3, 3\}, \{4, 4\}\big\}$ &  \\
    \hline
    $\big\{ \{2, 2\}, \{3, 3\} \big\}$ & $\big\{ \{2, 5\}, \{3, 6\}, \{4, 7\}\big\}$ \\
   $\big\{ \{2, 2\}, \{5, 5\} \big\}$ & \\
   \hline
    $\big\{ \{1, 1\}, \{2, 2\} \big\}$ & $\big\{ \{1, 2\}, \{1, 3\}, \{1, 4\}\big\}$ \\
    $\big\{ \{2, 2\}, \{3, 3\} \big\}$ &  \\
       \hline
 $\big\{ \{1, 1\}, \{2, 2\}\big\}$ & $\big\{ \{1, 2\}, \{1, 3\}, \{1, 7\}\big\}$ \\
  $\big\{ \{2, 2\}, \{3, 3\} \big\}$ &  \\
 \hline
  $\big\{ \{1, 1\}, \{2, 2\} \big\}$ & $\big\{ \{1, 2\}, \{1, 6\}, \{2, 6\}\big\}$ \\
  $\big\{\{1, 1\}, \{7, 7\}\big\}$ &  \\
  \hline
  $\big\{ \{1, 1\}, \{2, 2\}\big\}$ & $\big\{ \{1, 2\}, \{1, 6\}, \{6, 7\}\big\}$ \\
  $\big\{ \{1, 1\}, \{4, 4\} \big\}$ &  \\
  $\big\{ \{1, 1\}, \{6, 6\}\big\}$ &  \\
  $\big\{ \{1, 1\}, \{7, 7\}\big\}$ &  \\
  $\big\{ \{2, 2\}, \{4, 4\}\big\}$ &  \\
  $\big\{ \{2, 2\}, \{7, 7\}\big\}$ &  \\
  \hline
  $\big\{ \{1, 1\}, \{2, 2\} \big\}$ & $\big\{ \{1, 2\}, \{1, 6\}, \{1, 7\}, \{2, 6\}\big\}$ \\
  $\big\{ \{1, 1\}, \{7, 7\} \big\}$ & \\
  $\big\{ \{2, 2\}, \{6, 6\} \big\}$ & \\ 
  $\big\{\{2, 2\}, \{7, 7\}\big\}$ &  \\
  \hline  
$\big\{ \{1, 1\}, \{2, 2\}\big\}$ & $\big\{ \{1, 2\}, \{1, 6\}, \{2, 7\}, \{6, 7\}\big\}$ \\
  $\big\{\{1, 1\}, \{4, 4\}\big\}$ & \\ 
   $\big\{\{1, 1\}, \{7, 7\}\big\}$ & \\ 
  \hline 
  $\big\{ \{1, 1\}, \{2, 2\} \big\}$ & $\big\{ \{1, 2\}, \{1, 6\}, \{1, 7\}, \{2, 6\}, \{2, 7\}\big\}$ \\
  $\big\{ \{1, 1\}, \{6, 6\} \big\}$ &  
  \\
  $\big\{ \{6, 6\}, \{7, 7\} \big\}$ &  
  \\
  \hline  
  $\big\{ \{3, 3\}, \{4, 4\} \big\}$ & $X_{\ell^c_{12}}$ \\
  \hline 
  $ \big\{ \{1, 1\}, \{2, 2\}\big\} $ & $P_{\{1, 2, 3\}}$ \\
  \hline 
\caption{\label{gnsLastcase2} $T - X = \big\{\{i, i\}, \{j, j\}\big\}$ for $i, j \in I$ distinct.}
\end{longtable}
 
  
 Taking into account Corollary~\ref{le_conii2}, 
 the only difficulty in   assembling this table lies in avoiding collinear sets.
 We illustrate some examples.
 For $S=\big\{ \{1, 2\}\big\}$, any $S_{\{i,j\}}$ with $i,j\in I-\{5\}$ is generalised nice. The cases to be distinguished are: both $i,j\in\{1,2\}$; only one index (we can assume the index $1$) belonging to $\{1,2\}$; and $i,j\notin\{1,2,5\}$. In the last case, we have two possibilities, according to $i*j=5=1*2$  or $i*j\ne5$. 
 In this way, $S_{\{3, 4\}}$  is not collinear to $S_{\{3,6\}}$.
 
  For each possible $S$ the discussion is slightly different, always keeping in mind that a choice of 3 indices (each non-zero and together not constituting a line)  describes a unique collineation. As a second example, for $S=\big\{ \{1, 2\}, \{6, 7\}\big\}$, any $S_{\{i,j\}}$ with $i,j\in I-\{5\}$ is generalised nice. If $i\in\{1,2,6,7\}$, we can assume that $i=1$, and then we have to distinguish  between $j=2$ (the companion of 1), $j\in\{6,7\},$ and $j\in\{3,4\}$.
   This gives just 3 possibilities with $i=1$. Otherwise, both $i,j\in\{3,4\}$ and we  find only  one more equivalence class:  $S_{\{3,4\}}$. 

  We demonstrate one final example, $S=\big\{ \{1, 2\}, \{1, 3\}\big\}$, and leave the rest to the reader. If a collineation $\sigma$ sends $S_{\{i,j\}}$ to $S_{\{i',j'\}}$ (these sets are generalised nice if and only if  $i,j,i',j'\ne 5,6$), then  $\sigma(1)=1$ and $\sigma$ preserves $\{2,3\}$. Ergo, either $\sigma=\id$, or $\sigma(2)=3$ and $\sigma(3)=2$. In either case, $\sigma(4)=4$ (since $4=1*2*3$) and $\sigma(7)=7$. This means that $S_{\{ 1,2\}}$ ($\sim_c S_{\{1,3 \}}$), $S_{\{1,4 \}},$ and $S_{\{1,7 \}}$ are pairwise not collinear. Also, if at least one of the indices $i,j$ is in $\{2,3\}$ we can assume $i=2,$ and then $j=3,4,7$. This leaves three pairwise not collinear  generalised nice sets (since $3\in I_S-J_S$, while $4,7\in I-I_S$ but
  $7=2*3$ belongs to $ (I_S-J_S)*(I_S-J_S)$ and 4 does not). Finally, if $i,j\notin\{1,2,3\}$, then both $i,j\in\{4,7\}$, giving 
  rise to only one more generalised nice set; the seventh possibility attached to $S$. 
 The remaining cases do not present any additional difficulties.
 \smallskip
 
In Table \ref{gnsLastcase2} we   covered one case where $T - X$ has cardinal 2. 
The only other possibility for $T - X$ is
$\big\{\{0, 0\}, \{0, i\}\big\}$, for any $i \in I$. Now Theorem~\ref{todo} implies:

\begin{cor}\label{le_0i00}
  Let   $S\subseteq X$ be a   generalised nice set, 
and $i \in I$.  The set 
$S'_{\{  i\}}= S\cup \big\{\{0, 0\},\{0 ,  i\}\big\}$   is a generalised nice set 
if and only if    $i\notin I_S$.
\end{cor}

In particular,
there does not exist $i \in I$ such  that  $S \cup \big\{\{0, 0\}, \{0, i\}\big\} $ is generalised nice, for $S$ any of the sets below. In fact, for all these cases $I_S=I$ according to Table~\ref{IsyJs}.
\begin{equation}\label{losdeI_S=I}
 \begin{array}{c}
 \big\{\{2, 5\}, \{3, 6\}, \{4, 7\}\big\}; \ 
 \big\{\{1, 2\}, \{1, 3\}, \{1, 4\}\big\};\ 
 \big\{\{1, 2\}, \{1, 3\}, \{1, 7\}\big\};\\
 \big\{\{1, 2\}, \{1, 6\}, \{1, 7\}, \{2, 6\}\big\};\ 
 \big\{\{1, 2\}, \{1, 6\}, \{1, 7\}, \{2, 6\}, \{2, 7\} \big\};\ 
 X_{\ell^c_{12}};  \ 
 P_{\{1, 2, 3\}}.
 \end{array}
 \end{equation}
 We obtain fewer cases now, since $I_S$ contains strictly $J_S$:

\smallskip

\begin{longtable} [c] { | c | c|  }
    \hline
    \multicolumn{2} { | c | }{{\bf Gns $T$ such that both $T \cap X$ and $T - X$ are non-empty gns}} \\
   \hline
  $T - X$ & $T \cap X$ \\
    \hline
    \endfirsthead
    \hline
    $T - X$ & $T \cap X$ \\
    \hline
    \endhead
    \hline
    \endfoot
    \hline
    \endlastfoot
    $\big\{ \{0, 0\}, \{0, 3\} \big\}$ & $\big\{ \{1, 2\}\big\}$ \\
    \hline
    $\big\{ \{0, 0\}, \{0, 4\} \big\}$ & $\big\{ \{1, 2\}, \{1, 3\}\big\}$ \\
    $\big\{ \{0, 0\}, \{0, 7\}\big\}$ &  \\
   \hline 
   $\big\{ \{0, 0\}, \{0, 3\} \big\}$ & $\big\{ \{1, 2\}, \{6, 7\}\big\}$ \\
    \hline
  $\big\{ \{0, 0\}, \{0, 7\} \big\}$ & $\big\{ \{1, 2\}, \{1, 6\}, \{2, 6\}\big\}$ \\
  \hline
  $\big\{ \{0, 0\}, \{0, 4\}\big\}$ & $\big\{ \{1, 2\}, \{1, 6\}, \{6, 7\}\big\}$ \\
  \hline
 $\big\{ \{0, 0\}, \{0, 4\}\big\}$ & $\big\{ \{1, 2\}, \{1, 6\}, \{2, 7\}, \{6, 7\}\big\}$ \\
  \hline
  \caption{\label{gnsLastcase3} $T - X = \big\{\{0, 0\}, \{0, i\}\big\}$ for $i \in I$.}
\end{longtable}

 The possibilities, up to collineation, are easy to analyse by keeping in mind Corollary~\ref{le_0i00}. 
For  $S=\big\{ \{1, 2\}\big\}$, $S'_{\{ i\}}$  is generalised nice as long as $i\ne 1,2,5$. Here the four sets  $S'_{\{ 3\}}$, $S'_{\{ 4\}}$, $S'_{\{6 \}},$ and $S'_{\{7 \}}$  are all collinear, since $\tilde\sigma_{12k}(S'_{\{ 3\}})=S'_{\{ k\}}$ for any $k=4,6,7$. 

In contrast, for the case $S=\big\{ \{1, 2\},\{1, 3\}\big\}$, only two sets are generalised nice; $S'_{\{ 4\}}$  and $S'_{\{ 7\}}$.
They are, however, not collinear because any collineation $\sigma\ne\id$ with $\tilde\sigma(S)=S$ must also satisfy $\sigma(1)=1$ and $\sigma(2)=3$, which implies $\sigma(4)=4$.

Third, we can add $\big\{\{0, 3\},\{0 ,  0\}\big\}$ and $\big\{\{0, 4\},\{0 ,  0\}\big\}$ to $S= \big\{ \{1, 2\},\{6, 7\}\big\}$ 
but the obtained sets are collinear (an appropriate collineation would be $ \sigma_{124}$). 
Finally, as $|I-I_S|=1$ in remaining cases, there is only one index such that $S'_{\{ i\}}$  is generalised nice.

\subsection{$T - X$ of cardinality 3} 

As per in Table \ref{gnsEmptyX} there are several different types of generalised nice sets contained in $X_0 - X.$ Namely: 

\begin{itemize}
\item[-] $\big\{ \{i, i\}, \{j, j\}, \{k, k\} \big\},$ for $i, j, k \in I$ generative;
\item[-] $\big\{ \{i, i\}, \{j, j\}, \{i \ast j, i \ast j\} \big\},$ for $i, j \in I$ distinct;
\item[-] $\big\{ \{0, 0\}, \{0, i\}, \{0, j\} \big\},$ for $i, j \in I$ distinct;
\item[-] $\big\{ \{0, 0\}, \{0, i\}, \{i, i\} \big\},$ for $i \in I$.
\end{itemize}

We begin by   recalling what Theorem~\ref{todo} says in the first two cases,
%
\begin{cor}\label{le_conii3}
If $S\subseteq X$ is a   generalised nice set,  
and $i,j,k\in I$ are distinct, then  
$S_{\{i,j,k\}}=S\cup \big\{\{i, i\},\{j,j\},\{k,k\}\big\}$   is a generalised nice set 
if and only if        $ i,j,k\notin J_S$.
\end{cor}

From here, all the sets obtained as union of the two columns of the following table are generalised nice sets. 
We have  first considered the case where $i,j,k$ are generative.
\smallskip

\begin{longtable} [c] { | c | c|  }
    \hline
    \multicolumn{2} { | c | }{{\bf Gns $T$ such that both $T \cap X$ and $T - X$ are non-empty gns}} \\
   \hline
  $T - X$ & $T \cap X$ \\
    \hline
    \endfirsthead
    \hline
    $T - X$ & $T \cap X$ \\
    \hline
    \endhead
    \hline
    \endfoot
    \hline
    \endlastfoot
    $\big\{ \{1, 1\}, \{2, 2\}, \{3, 3\} \big\}$ & $\big\{ \{1, 2\}\big\}$ \\
    $\big\{ \{1, 1\}, \{3, 3\}, \{4, 4\} \big\}$ & \\
    $\big\{ \{3, 3\}, \{4, 4\}, \{6, 6\} \big\}$ & \\
    \hline
    $\big\{ \{1, 1\}, \{2, 2\}, \{3, 3\}  \big\}$ & $\big\{ \{1, 2\}, \{1, 3\}\big\}$ \\
    $\big\{ \{1, 1\}, \{2, 2\}, \{4, 4\} \big\}$ &  \\
    $\big\{ \{1, 1\}, \{2, 2\}, \{7, 7\}\big\}$ & \\
     $\big\{ \{2, 2\}, \{3, 3\}, \{4, 4\}\big\}$ & \\
    $\big\{ \{2, 2\}, \{4, 4\}, \{7, 7\} \big\}$ &  \\
    \hline 
    $\big\{ \{1, 1\}, \{2, 2\}, \{3, 3\} \big\}$ & $\big\{ \{1, 2\}, \{6, 7\}\big\}$ \\
    $\big\{ \{1, 1\}, \{2, 2\}, \{6, 6\}\big\}$ &  \\
    $\big\{ \{1, 1\}, \{3, 3\}, \{4, 4\}\big\}$ &  \\
    $\big\{ \{1, 1\}, \{3, 3\}, \{7, 7\}\big\}$ &  \\
    \hline
    $\big\{ \{2, 2\}, \{3, 3\}, \{4, 4\} \big\}$ & $\big\{ \{2, 5\}, \{3, 6\}, \{4, 7\}\big\}$ \\
    $\big\{ \{2, 2\}, \{3, 3\}, \{5, 5\} \big\}$ & \\
    \hline
    $\big\{ \{1, 1\}, \{2, 2\}, \{3, 3\} \big\}$ & $\big\{ \{1, 2\}, \{1, 3\}, \{1, 4\}\big\}$ \\
    $\big\{ \{2, 2\}, \{3, 3\}, \{4, 4\} \big\}$ &  \\ 
    \hline
    $\big\{ \{1, 1\}, \{2, 2\}, \{3, 3\} \big\}$ & $\big\{ \{1, 2\}, \{1, 3\}, \{1, 7\}\big\}$ \\
    \hline
    $\big\{ \{1, 1\}, \{2, 2\}, \{6, 6\}\big\}$ & $\big\{ \{1, 2\}, \{1, 6\}, \{2, 6\}\big\}$ \\
    $\big\{\{1, 1\}, \{2, 2\}, \{7, 7\}\big\}$ &  \\
  \hline
  $\big\{ \{1, 1\}, \{2, 2\}, \{4, 4\} \big\}$ & $\big\{ \{1, 2\}, \{1, 6\}, \{6, 7\}\big\}$ \\ 
  $\big\{\{1, 1\}, \{2, 2\}, \{6, 6\}\big\}$ & \\ 
   $\big\{\{1, 1\}, \{2, 2\}, \{7, 7\}\big\}$ & \\ 
    $\big\{\{1, 1\}, \{4, 4\}, \{6, 6\}\big\}$ & \\ 
     $\big\{\{2, 2\}, \{4, 4\}, \{7, 7\}\big\}$ & \\ 
  \hline
  $\big\{ \{1, 1\}, \{2, 2\}, \{6, 6\} \big\}$ & $\big\{ \{1, 2\}, \{1, 6\}, \{1, 7\}, \{2, 6\}\big\}$ \\
   $\big\{ \{1, 1\}, \{2, 2\}, \{7, 7\} \big\}$  & \\
  $\big\{\{2, 2\}, \{6, 6\}, \{7, 7\}\big\}$ &  \\
  \hline
$\big\{ \{1, 1\}, \{2, 2\}, \{4, 4\} \big\}$ & $\big\{ \{1, 2\}, \{1, 6\}, \{2, 7\}, \{6, 7\}\big\}$ \\ 
  $\big\{\{1, 1\}, \{2, 2\}, \{6, 6\}\big\}$ & \\ 
  \hline
  $\big\{ \{1, 1\}, \{2, 2\}, \{6, 6\} \big\}$ & $\big\{ \{1, 2\}, \{1, 6\}, \{1, 7\}, \{2, 6\}, \{2, 7\}\big\}$ \\
  $\big\{ \{1, 1\}, \{6, 6\}, \{7, 7\} \big\}$ &  
  \\
  \hline
  $\big\{ \{3, 3\}, \{4, 4\}, \{6, 6\} \big\}$ & $X_{\ell^c_{12}}$ \\
  \hline
  $\big\{ \{1, 1\}, \{2, 2\}, \{3, 3\} \big\}$ & $P_{\{1, 2, 3\}}$ \\
  \hline
\caption{\label{gnsLastcase4}   $T - X=\big\{ \{i, i\}, \{j, j\}, \{k, k\} \big\}$, $i, j, k  $ generative.  }
\end{longtable}
 
   The difficulty   again lies in being careful with possible collinear sets. Some examples follow.
 For  
  $S =\big\{ \{1, 2\}, \{1, 3\}\big\}$, $S_{\{i,j,k\}}$
  is generalised nice for any distinct $i,j,k\ne 5,6$. This gives   $\binom53=10$ possibilities but many of them are collinear. 
  Let us consider $k\ne i*j$.
  We can assume that $\{i,j,k\}\cap\{1,2,3\}$ is one of: $\{1,2,3\}$, $\{1,2\}$, $\{2,3\}$, $\{1\}$, or $\{2\}$. 
  We consider how many possibilities are related to these cases. 
  If $i=1$ and $j=2$ (with $k\ne3$), then $k$ can be $4$ or $7$. These two possibilities are not collinear because $7=2*3$ is fixed by any collineation preserving the set $\{2,3\}$.
  If $i=2$ and $j=3$, then $k\ne 1,5,6,7$ gives $k=4$. 
  If $i=1$, $j,k\ne 2,3,5,6$, the only possibilities for $j$ and $k$ are $4$ and $7$, but $1*4=7$. 
  If $i=2$, $j,k\ne 1,3,5,6$,  then $j=4$ and $k=7$. 
  Hence, there are 5 pairwise not collinear generalised nice sets $S_{\{i,j,k\}}$ with $k\ne i*j$. 

 Now we discuss another example: 
  $S=\big\{ \{1, 2\}, \{1, 6\}, \{1, 7\}, \{2, 6\}, \{2, 7\}\big\}$.   
  The generalised nice sets are $S_J$ for 
    $ J\subseteq I-J_S= \{1,2,6,7\}$.  Note that 
     the collineations   leave invariant the sets $ \{1,2\}$ and $\{6,7\}$, according to the number of occurrences of these indices in elements of $S$.
  So there are two not collinear $S_J$'s, accordingly to whether   the only index in $  \{1,2,6,7\}-J$ belongs either to $ \{1,2\}$ or to $\{6,7\}$.

 An easy example is $S =X_{\ell_{12}^c}$. Here $S_{\{i,j,k\}}$
  is generalised nice for any choice of $i,j,k\in\{3,4,6,7\}$. The four indices all play the same role, and so any
   $S_{\{i,j,k\}}$ is collinear to any $S_{\{i',j',k'\}}$. 
  The   final example, $S=P_{\{1,2,3\}}$ has $I - J_S=\{1,2,3\}$. So there is only one possibility, namely, $S_{\{1,2,3\}}$.

  The remaining cases can be similarly discussed (no new difficulties arise), and they are left to the reader.
  Besides, to   ensure we have not overlooked any collinear sets, we have employed computer assistance in checking the many cases
  (see Section~\ref{computer}).  
  \smallskip  
 
Suppose now that $T - X$ is of the form $\big\{\{i, i\}, \{j, j\}, \{i \ast j, i\ast j\}\}$, for $i, j \in I$ distinct. As before, Corollary~\ref{le_conii3} 
applies to   reveal that the following generalised nice sets, contained in $X$ (as per in Corollary~\ref{somegns}),  cannot be combined with $T-X$:
\begin{equation}\label{J_Scontieneline}
\begin{array}{c}
\big\{\{1, 2\}, \{1, 3\}, \{1, 4\}\big\};\ 
\big\{\{1, 2\}, \{1, 6\}, \{2, 6\}\big\};\ 
\big\{\{1, 2\}, \{1, 6\}, \{1, 7\}, \{2, 6\}\big\};\\ 
\big\{\{1, 2\}, \{1, 6\}, \{1, 7\}, \{2, 6\}, \{2, 7\} \big\};\ 
X_{\ell^c_{12}};\ 
P_{\{1, 2, 3\}}.
\end{array}
\end{equation}
The reason is that  
 there is a line $\ell  \subseteq  J_S$ for each of  these possible $S$'s   according to Table~\ref{IsyJs}.
 As any two lines intersect, there is no line $\ell_{ij}=\{i,j,i*j\}$ contained in $I-J_S$ and so $S_{\{i,j,i*j\}}$ is not generalised  nice.
 On the contrary, there is a line   contained in $I-J_S$
for  all the remaining sets in Corollary~\ref{somegns}. 
The discussion about the possible collineations is quite simple   now.
%
Only in case $S=\big\{ \{1, 2\}, \{1, 3\}\big\}$,  the two lines contained in $I-J_S=\{1,2,3,4,7\}$ satisfy 
$|\ell_{14}\cap I_S|=1$  while $|\ell_{23}\cap I_S|=2$, so that
  the sets $S_{\ell_{14}}$ and $S_{\ell_{23}}$
are not collinear.   Each of the remaining sets only gives arise to   one possibility up to collineations. 
Indeed, there is only one line in $I-J_S$ if $S=\big\{ \{1, 2\}, \{1, 3\}, \{1, 7\}\big\}$. 
For the sets $S$'s with $J_S=\{j\}$, there are  four lines in $I-J_S$, but $S\subseteq X^{(j)}$ and each pair in $S$ contains exactly
 one index of each line. 
Thus it is easy to find a collineation $\sigma$ with $\tilde\sigma(S)=S$ and $ \sigma(\ell)=\ell'$ for any $\ell,\ell'\subseteq I-J_S$.
Finally, for the   last two sets $S$'s,  there are two lines in $I-\{3,5\}$, namely, $ \ell_{14},\ell_{24}$,
 but $\tilde\sigma_{673}(S_{\ell_{14}})=S_{\ell_{24}}$. 
%
We collect the resulting generalised nice sets in the following table:

\smallskip

\begin{longtable} [c] { | c | c|  }
    \hline
    \multicolumn{2} { | c | }{{\bf Gns $T$ such that both $T \cap X$ and $T - X$ are non-empty gns}} \\
   \hline
  $T - X$ & $T \cap X$ \\
    \hline
    \endfirsthead
    \hline
    $T - X$ & $T \cap X$ \\
    \hline
    \endhead
    \hline
    \endfoot
    \hline
    \endlastfoot
    $\big\{ \{1, 1\}, \{3, 3\}, \{6, 6\} \big\}$ & $\big\{ \{1, 2\}\big\}$ \\
    \hline
    $\big\{ \{1, 1\}, \{4, 4\}, \{7, 7\} \big\}$ & $\big\{ \{1, 2\}, \{1, 3\}\big\}$ \\
    $\big\{ \{2, 2\}, \{3, 3\}, \{7, 7\} \big\}$ &  \\
   \hline 
    $\big\{ \{1, 1\}, \{3, 3\}, \{6, 6\} \big\}$ & $\big\{ \{1, 2\}, \{6, 7\}\big\}$ \\
    \hline
    $\big\{ \{2, 2\}, \{3, 3\}, \{7, 7\} \big\}$ & $\big\{ \{2, 5\}, \{3, 6\}, \{4, 7\}\big\}$ \\
   \hline
   $\big\{ \{2, 2\}, \{3, 3\}, \{7, 7\} \big\}$  & $\big\{ \{1, 2\}, \{1, 3\}, \{1, 7\}\big\}$ \\
  \hline
  $\big\{ \{1, 1\}, \{4, 4\}, \{7, 7\} \big\}$ & $\big\{ \{1, 2\}, \{1, 6\}, \{6, 7\}\big\}$ \\
  \hline
$\big\{ \{1, 1\}, \{4, 4\}, \{7, 7\} \big\}$ & $\big\{ \{1, 2\}, \{1, 6\}, \{2, 7\}, \{6, 7\}\big\}$ \\
  \hline
  \caption{\label{gnsLastcase5} $T - X = \big\{\{i, i\}, \{j, j\}, \{i \ast j, i \ast j \}\big\}$ for $i\ne  j  $.}
\end{longtable}

 
Next, we assume that our $T - X$ (the part of $T$ contained in $X_0 - X$) is of the form $\big\{\{0, 0\}, \{0, i\}, \{0, j\}\}$, for some distinct $i, j \in I$. 
  Again, Theorem~\ref{todo} makes it possible to rule out a lot of cases, since $|I_S|\ge6$ whenever $S$ belongs to 
  $$
  \begin{array}{c}
  \big\{\{2, 5\}, \{3, 6\}, \{4, 7\}\big\}; \ 
  \big\{\{1, 2\}, \{1, 3\}, \{1, 4\}\big\};\ 
   \big\{\{1, 2\}, \{1, 3\}, \{1, 7\}\big\};\\
   \big\{\{1, 2\}, \{1, 6\}, \{2, 6\}\big\};\ 
    \big\{\{1, 2\}, \{1, 6\}, \{6, 7\}\big\}; \ 
    \big\{\{1, 2\}, \{1, 6\}, \{1, 7\}, \{2, 6\}\big\}; \\
   \big\{\{1, 2\}, \{1, 6\}, \{2, 7\}, \{6, 7\}\big\}; \ 
   \big\{\{1, 2\}, \{1, 6\}, \{1, 7\}, \{2, 6\}, \{2, 7\} \big\}; \ 
   X_{\ell^c_{12}} ; \ 
     P_{\{1, 2, 3\}}.
     \end{array}
  $$

This theorem also tells us that there is only one set of the form $\big\{\{0, 0\}, \{0, i\}, \{0, j\}\}$
to be combined with $S=\big\{ \{1, 2\}, \{1, 3\}\big\}$, since $i,j$ has to belong to $ I-I_S= \{4,7\}$. The situation is the same for $S=\big\{ \{1, 2\}, \{6,7\}\big\}$. On the other hand  $\big\{\{0, 0\}, \{0, i\}, \{0, j\},\{1, 2\}\}  $ is generalised  nice if and only if $i,j\in\{3,4,6,7\}$. There are just two not collinear possibilities now: if $i*j=1*2$ and if $i*j\in\{1,2\}$.
To summarize,  we obtain the following generalised nice sets up to collineations:
%

\begin{longtable} [c] { | c | c|  }
    \hline
    \multicolumn{2} { | c | }{{\bf Gns $T$ such that both $T \cap X$ and $T - X$ are non-empty gns}} \\
   \hline
  $T - X$ & $T \cap X$ \\
    \hline
    \endfirsthead
    \hline
    $T - X$ & $T \cap X$ \\
    \hline
    \endhead
    \hline
    \endfoot
    \hline
    \endlastfoot
    $\big\{ \{0, 0\}, \{0, 3\}, \{0, 4\} \big\}$ & $\big\{ \{1, 2\}\big\}$ \\
     $\big\{ \{0, 0\}, \{0, 3\}, \{0, 7\} \big\}$ &  \\
    \hline
    $\big\{ \{0, 0\}, \{0, 4\}, \{0, 7\} \big\}$ & $\big\{ \{1, 2\}, \{1, 3\}\big\}$ \\
   \hline 
   $\big\{ \{0, 0\}, \{0, 3\}, \{0, 4\} \big\}$ & $\big\{ \{1, 2\}, \{6, 7\}\big\}$ \\
    \hline
  \caption{\label{gnsLastcase6} $T - X \subseteq X_F$, $|T-X|=3$.}
\end{longtable}

To finish the cardinality 3 section, we assume now that $T - X=\big\{\{0, 0\}, \{0, i\}, \{i, i\} \big\}$, for some $i \in I$. 
  Any generalised nice set $S$ with $I_S=I$  cannot be combined with $T-X.$ That is, exactly
those listed in Eq.~\eqref{losdeI_S=I}.
 
From Theorem~\ref{todo}(c),  $\tilde S_i=S\cup \big\{\{0, 0\}, \{0, i\}, \{i, i\}  \big\}$ is generalised nice if and only if $i\notin I_S$,  
for any of the remaining sets $S$ in Corollary~\ref{somegns}.
Thus $\big\{\{0, 0\}, \{0, i\}, \{i, i\},\{1,2\} \big\}$ is generalised  nice for any $i\in\{3,4,6,7\}$, although   the four possibilities are clearly collinear. 
The same happens for $\big\{\{0, 0\}, \{0, i\}, \{i, i\},\{1,2\} ,\{6,7\} \big\}$ and $i=3,4$. On the contrary,  the only two generalised nice sets in the form  $\big\{\{0, 0\}, \{0, i\}, \{i, i\},\{1,2\} ,\{1,3\} \big\}$, obtained for $i=4$ and $7$, are not collinear. 
Lastly, for the three remaining sets with $|S|\ge3$, we have $|I_S|=6$ and only one index $i$ that makes $\tilde S_i$  generalised  nice.

\begin{longtable} [c] { | c | c|  }
    \hline
    \multicolumn{2} { | c | }{{\bf Gns $T$ such that both $T \cap X$ and $T - X$ are non-empty gns}} \\
   \hline
  $T - X$ & $T \cap X$ \\
    \hline
    \endfirsthead
    \hline
    $T - X$ & $T \cap X$ \\
    \hline
    \endhead
    \hline
    \endfoot
    \hline
    \endlastfoot
    $\big\{ \{0, 0\}, \{0, 3\}, \{3, 3\} \big\}$ & $\big\{ \{1, 2\}\big\}$ \\
    \hline
    $\big\{ \{0, 0\}, \{0, 4\}, \{4, 4\} \big\}$ & $\big\{ \{1, 2\}, \{1, 3\}\big\}$ \\
    $\big\{ \{0, 0\}, \{0, 7\}, \{7, 7\} \big\}$ &  \\
   \hline 
   $\big\{ \{0, 0\}, \{0, 3\}, \{3, 3\} \big\}$ & $\big\{ \{1, 2\}, \{6, 7\}\big\}$ \\
    \hline
    $\big\{ \{0, 0\}, \{0, 7\}, \{7, 7\} \big\}$ & $\big\{ \{1, 2\}, \{1, 6\}, \{2, 6\}\big\}$ \\
    \hline
    $\big\{ \{0, 0\}, \{0, 4\}, \{4, 4\} \big\}$ & $\big\{ \{1, 2\}, \{1, 6\}, \{6, 7\}\big\}$ \\
     \hline
    $\big\{ \{0, 0\}, \{0, 4\}, \{4, 4\} \big\}$ & $\big\{ \{1, 2\}, \{1, 6\}, \{2, 7\}, \{6, 7\}\big\}$ \\
    \hline
 \caption{\label{gnsLastcase7} $T - X = P_{\{0,i,i\}}$ for $i \in I$.}
\end{longtable}


\subsection{$T - X$ of cardinality 4} 

As per   Table \ref{gnsEmptyX} there are four different types of generalised nice set contained in $X_0 - X$; namely: 
\begin{itemize}
\item[-] $\big\{ \{i, i\}, \{j, j\}, \{k, k\}, \{l,l\} \big\},$ for $i, j, k,l \in I$ such that any three are generative;
\item[-] $\big\{ \{i, i\}, \{j, j\}, \{i \ast j, i \ast j\}, \{k, k\} \big\},$ for $i, j \in I$ distinct and $k \notin \ell_{ij}$;
\item[-] $\big\{ \{0, 0\}, \{0, i\}, \{0, j\}, \{0, k\} \big\},$ for $i, j, k \in I$ generative;
\item[-] $\big\{ \{0, 0\}, \{0, i\}, \{0, j\}, \{0, i \ast j\} \big\},$ for $i, j \in I$ distinct.
\end{itemize}

We study again which of them can be added to the generalised nice sets in Corollary~\ref{somegns}.

\begin{cor} \label{varioscardinal4}
 Let  $S\subseteq X$ be a  non-empty generalised nice set. 
 Then 
\begin{itemize}
\item[\rm(a)] $S\cup \big\{ \{i, i\}, \{j, j\}, \{k, k\}, \{l,l\} \big\},$ for $i, j, k,l \in I$ distinct, 
is a generalised nice set if and only if  $\{i,j,k,l\}\cap J_S=\emptyset$.
\item[\rm(b)] $S\cup \big\{ \{0, 0\}, \{0, i\}, \{0, j\}, \{0, k\} \big\},$ for $i, j, k \in I$ generative,
is a generalised nice set if and only if $\{i,j,k\}\cap I_S=\emptyset$. In particular, $|S|=1$.
\item[\rm(c)]   $ S\cup \big\{ \{0, 0\}, \{0, i\}, \{0, j\}, \{0, i*j\} \big\},$ for $i, j \in I$ distinct,
is a generalised nice set if and only if     $\{i,j,i*j \}=I_S$.  In particular, $|S|=1$.
\end{itemize}
\end{cor}

\begin{proof}
Theorem~\ref{todo} gives immediately (a).  

Case (b) is clear because $J=\{i,j,k\}$   contains no line, so the condition for $S'_J$ to be generalised  
nice is $\ell_{ab}\cap J=\emptyset$ for any $\{a, b\}\in S$. In other words, $J\cap I_S=\emptyset$. This situation is not possible if $S$ has at least two elements, since $|I-I_S|\le2$.

For item (c), $J=\{i,j,i*j\}=\ell_{ij}$ and we know that   $J=I_S$ is a sufficient condition to ensure 
 that $S'_J$ is generalised nice, due to 
Theorem~\ref{todo}(ii). 
Conversely, assume that $S'_J$ is generalised nice. If $\{a,b\}\in S$, Theorem~\ref{todo}(b) asserts that either $\ell_{ab}\cap \ell_{ij}=\emptyset$ (a contradiction) or $\ell_{ab}=\ell_{ij}=J$. Thus, $\{a,b,a*b\}\subseteq J$ for any $\{a,b\}\in S$; that is, $I_S\subseteq J$. In fact, they are equal because $S\ne\emptyset$. Besides, Table~\ref{IsyJs} says that $|I_S|=3$ if and only if $|S|=1$.
\end{proof}

Now we obtain the  list of related generalised nice sets. We begin by discussing the situation wherein $T-X=\big\{ \{i, i\}, \{j, j\}, \{k, k\}, \{l, l\} \big\}$,
and any three of $\{i,j,k,l\}$ forms a generative triplet.
We can rule out a few cases:
$$
\big\{ \{1, 2\}, \{1, 3\}, \{1, 7\}\big\},
$$
since  $I-J_S= \{1,2,3,7\}$ contains the line $\ell_{23}$;
and
$$
P_{\{1, 2, 3\}} = \big\{\{1, 2\}, \{1, 3\}, \{1, 7\}, \{2, 3\}, \{2, 6\}, \{3, 5\}\big\},
$$
 since $|I-J_S|=3$.  
Keeping  Corollary~\ref{varioscardinal4} in mind,  
we obtain the following new generalised nice sets.

\smallskip

\begin{longtable} [c] { | c | c|  }
    \hline
    \multicolumn{2} { | c | }{{\bf Gns $T$ such that both $T \cap X$ and $T - X$ are non-empty gns}} \\
   \hline
  $T - X$ & $T \cap X$ \\
    \hline
    \endfirsthead
    \hline
    $T - X$ & $T \cap X$ \\
    \hline
    \endhead
    \hline
    \endfoot
    \hline
    \endlastfoot
    $\big\{ \{1, 1\}, \{2, 2\}, \{3, 3\}, \{4, 4\} \big\}$ & $\big\{ \{1, 2\}\big\}$ \\
    $\big\{ \{3, 3\}, \{4, 4\}, \{6, 6\}, \{7, 7\} \big\}$ & \\
    \hline
    $\big\{ \{1, 1\}, \{2, 2\}, \{3, 3\}, \{4, 4\} \big\}$ & $\big\{ \{1, 2\}, \{1, 3\}\big\}$ \\
   \hline 
   $\big\{ \{1, 1\}, \{2, 2\}, \{3, 3\}, \{4, 4\} \big\}$ & $\big\{ \{1, 2\}, \{6, 7\}\big\}$ \\
    $\big\{ \{1, 1\}, \{2, 2\}, \{6, 6\}, \{7, 7\} \big\}$ & \\
    \hline
    $\big\{ \{2, 2\}, \{3, 3\}, \{5, 5\}, \{6, 6\} \big\}$ & $\big\{ \{2, 5\}, \{3, 6\}, \{4, 7\}\big\}$ \\
    \hline
    $\big\{ \{1, 1\}, \{2, 2\}, \{3, 3\}, \{4, 4\} \big\}$ & $\big\{ \{1, 2\}, \{1, 3\}, \{1, 4\}\big\}$ \\
    \hline
    $\big\{ \{1, 1\}, \{2, 2\}, \{6, 6\}, \{7, 7\} \big\}$ & $\big\{ \{1, 2\}, \{1, 6\}, \{2, 6\}\big\}$ \\
    \hline
    $\big\{ \{1, 1\}, \{2, 2\}, \{6, 6\}, \{7, 7\} \big\}$ &   $\big\{ \{1, 2\}, \{1, 6\}, \{6, 7\}\big\}$\\
    \hline
   $\big\{ \{1, 1\}, \{2, 2\}, \{6, 6\}, \{7, 7\} \big\}$ & $\big\{ \{1, 2\}, \{1, 6\}, \{1, 7\}, \{2, 6\}\big\}$ \\ 
   \hline
   $\big\{ \{1, 1\}, \{2, 2\}, \{6, 6\}, \{7, 7\} \big\}$ &  $\big\{ \{1, 2\}, \{1, 6\}, \{2, 7\}, \{6, 7\}\big\}$\\ 
    \hline
   $\big\{ \{1, 1\}, \{2, 2\}, \{6, 6\}, \{7, 7\} \big\}$ & $\big\{ \{1, 2\}, \{1, 6\}, \{1, 7\}, \{2, 6\}, \{2, 7\}\big\}$ \\ 
 \hline
 $\big\{ \{3, 3\}, \{4, 4\}, \{6, 6\}, \{7, 7\} \big\}$ & $X_{\ell^c_{12}}$ \\ 
 \hline
  \caption{\label{gnsLastcase8} $T - X = \big\{ \{k, k\}:k\notin\ell \big\}$ for $\ell$ a line.}
\end{longtable}

According to Corollary~\ref{varioscardinal4}, if $J$ is the complementary of a line, and $S\subseteq X$ is generalised, $S_J$ is generalised if and only if $J\cap J_S=\emptyset$, that is, if $J_S$ is contained in the line $I-J$.
As usual,  the only caution to be taken  is to avoid  the possible collineations.
For $S=\big\{ \{1, 2\}\big\}$, there are 3 lines containing the index $5$, but only two equivalence classes: $S_{I-\ell_{12}}$ and $S_{I-\ell_{34}}\sim_c S_{I-\ell_{67}}$. The case $S=\big\{ \{1, 2\},\{6,7\}\big\}$ is quite similar, since again $J_S=\{5\}$, although now $S_{I-\ell_{12}}\sim_c S_{I-\ell_{67}}$, while  $S_{I-\ell_{34}}$ is not collinear to them. The third and last case such that $J_S$ is contained in more than one line (that is, $|J_S|=1$) is 
$S=\big\{ \{2, 5\}, \{3, 6\}, \{4, 7\}\big\}$, but of course only one equivalence class appears: $S_{I-\ell_{25}}\sim_c S_{I-\ell_{36}}\sim_c S_{I-\ell_{47}}$.
Each of the remaining cases   gives arise to only one new generalised nice set, since $|J_S|\ge 2$ and only one line can contain $J_S$. 
\smallskip

Next, let $T - X$ be of the form $\big\{ \{i, i\}, \{j, j\}, \{i \ast j, i \ast j\}, \{k, k\} \big\},$ for $i, j \in I$ distinct and $k \notin \ell_{ij}$.
No set in Eq.~\eqref{J_Scontieneline} can be combined with such $T - X$
since there is a line contained in $ J_S$, and hence no line is contained in $I-J_S$.
  For the remaining sets $S$'s, we collect the possible generalised nice sets obtained in this way:  
  
  %

\begin{longtable} [c] { | c | c|  }
    \hline
    \multicolumn{2} { | c | }{{\bf Gns $T$ such that both $T \cap X$ and $T - X$ are non-empty gns}} \\
   \hline
  $T - X$ & $T \cap X$ \\
    \hline
    \endfirsthead
    \hline
    $T - X$ & $T \cap X$ \\
    \hline
    \endhead
    \hline
    \endfoot
    \hline
    \endlastfoot
    $\big\{ \{1, 1\}, \{3, 3\}, \{6, 6\}, \{2, 2\} \big\}$ & $\big\{ \{1, 2\}\big\}$ \\
    $\big\{ \{1, 1\}, \{3, 3\}, \{6, 6\}, \{7, 7\} \big\}$ & \\
    \hline
    $\big\{ \{2, 2\}, \{3, 3\}, \{7, 7\}, \{1, 1\} \big\}$ & $\big\{ \{1, 2\}, \{1, 3\}\big\}$ \\
    $\big\{ \{1, 1\}, \{4, 4\}, \{7, 7\}, \{2, 2\} \big\}$ &  \\
    $\big\{ \{2, 2\}, \{3, 3\}, \{7, 7\}, \{4, 4\} \big\}$ &  \\
   \hline 
    $\big\{ \{1, 1\}, \{3, 3\}, \{6, 6\}, \{2, 2\} \big\}$ & $\big\{ \{1, 2\}, \{6, 7\}\big\}$ \\
    $\big\{ \{1, 1\}, \{3, 3\}, \{6, 6\}, \{4, 4\} \big\}$ & \\
    \hline
    $\big\{ \{3, 3\}, \{4, 4\}, \{5, 5\}, \{2, 2\} \big\}$ & $\big\{ \{2, 5\}, \{3, 6\}, \{4, 7\}\big\}$ \\
    \hline 
    $\big\{ \{2, 2\}, \{3, 3\}, \{7, 7\}, \{1, 1\} \big\}$ & $\big\{ \{1, 2\}, \{1, 3\}, \{1, 7\}\big\}$ \\
     \hline
    $\big\{ \{2, 2\}, \{4, 4\}, \{6, 6\}, \{1, 1\} \big\}$ & $\big\{ \{1, 2\}, \{1, 6\}, \{6, 7\}\big\}$ \\
    $\big\{ \{1, 1\}, \{4, 4\}, \{7, 7\}, \{2, 2\} \big\}$ &  \\
    \hline
   $\big\{ \{2, 2\}, \{4, 4\}, \{6, 6\}, \{1, 1\} \big\}$ & $\big\{ \{1, 2\}, \{1, 6\}, \{2, 7\}, \{6, 7\}\big\}$ \\ 
 \hline
  \caption{\label{gnsLastcase9} $T - X = \ell_{ij}\cup \big\{\{k, k\} \big\},$ for   $k \notin \ell_{ij}$.}
\end{longtable}

  We have omitted here most of the discussions on collineations.  As a sample, if 
$S=\big\{ \{1, 2\}, \{1, 6\}, \{6, 7\}\big\}$, then  there are 4 possibilities of $J=\ell_{ij}\cup\{k\}\subseteq I-J_S$, namely:
$\ell_{24}\cup\{1\}$,
$\ell_{24}\cup\{7\}$, $\ell_{14}\cup\{2\}$ and $\ell_{14}\cup\{6\}$. The first and the  fourth  related generalised nice sets $S_J$'s are 
collinear, as well as the second and the third (in both cases we can use $\sigma_{673}$).  
These two equivalence classes are not collinear because any such collineation $\sigma$ would preserve $S$ and in particular $\{1,6\}$, whose indices appear twice each as part of elements in $S$. Thus $\sigma$  cannot send $\{2,4,6,1\}$ to $\{1,4,7,2\}$.
  The discussion changes a little bit depending on the particular $S$. For instance, in the case $S=\big\{ \{1, 2\} \big\}$, there are 4 lines contained in $I-J_S=I-\{5\}$ and there is no problem in fixing any of these, for instance, $\ell_{13}$. To produce sets which are not collinear one must consider for the fourth element $k,$ either the other element of the pair $\{1, 2\} $ not considered in the initial line, or one different from $1$ and $2$. This leads to two possibilities.  
  Although these have been the guidelines for our analysis in the remaining examples too,  we have also checked the work using computer assistance. The algorithm has  been described in detail  in Section~\ref{computer}.
\smallskip

Let us apply Corollary~\ref{varioscardinal4} to  $T - X=\big\{ \{0, 0\}, \{0, i\}, \{0, j\}, \{0, k\} \big\},$ for pairwise distinct $i,j,k\in I$ (whether generative or not),
to finish the cardinality 4 case.

\begin{prop} \label{noolvidardel4}
If $\emptyset\ne S\subseteq X$ and $S'_J$  are generalised nice sets such that $J\subseteq I$ has cardinal 3,
 then $|S|=1$ and $S'_J$ is collinear to one of these sets:
\begin{itemize}
\item[\rm(i)] 
$\big\{ \{0, 0\}, \{0, 3\}, \{0, 4\}, \{0, 7\}, \{1, 2\}\big\}$;
\item[\rm(ii)] 
$ \langle P_{\{0, 1, 2\}}\rangle = \big\{\{0, 0\}, \{0, 1\}, \{0, 2\}, \{0, 5\}, \{1, 2\}\big\}$.
 \end{itemize}
\end{prop}

\begin{proof}
We can assume  $S=\big\{  \{1, 2\}\big\}$ by Corollary~\ref{varioscardinal4}, and so  $I_S=\{1,2,5\}$. Any choice of 3 indices $i,j,k$ (necessarily generative) in $I-I_S=\{3,4,6,7\}$ would give a generalised nice set, but the four choices would get collinear sets. 

If $k=i*j$, the only possibility is $\ell_{ij}=I_S$.   
\end{proof}

\subsection{$T - X$ of cardinality 5} 

As per in Table \ref{gnsEmptyX} there are three different types of generalised nice sets contained in $X_0 - X$; namely: 
\begin{itemize}
\item[-] $\big\{ \{i, i\}, \{j, j\}, \{i \ast j, i \ast j\}, \{k, k\}, \{l , l\} \big\},$ for $k, l  \in I$ distinct such that $k \ast l= i \ast j$;
%
%

\item[-] $\big\{ \{0, 0\}, \{0, i\}, \{0, j\}, \{0, k\}, \{0,  l\} \big\},$ where $i, j, k, l \in I$ are such that 
any triplet among them is generative;  
\item[-] $\big\{ \{0, 0\}, \{0, i\}, \{0, j\}, \{0, i \ast j\}, \{0, k\} \big\},$ for $i, j, k \in I$ generative.
\end{itemize}

\medskip

Suppose first that $T - X\subseteq X_E^*$ with  $| T-X|=5$.   
 (Note that any 5 distinct indices in $I$ are the union of two lines.) 
Then $T\cap X$  
  cannot be any of the sets in Eq.~\eqref{J_Scontieneline} nor $\big\{ \{1, 2\}, \{1, 3\}, \{1, 7\}\big\};$
since these 7 sets are exactly  those ones with $|J_S|\ge3$.
All the remaining cases give rise to generalised nice sets: 

\smallskip

\begin{longtable} [c] { | c | c|  }
    \hline
    \multicolumn{2} { | c | }{{\bf Gns $T$ such that both $T \cap X$ and $T - X$ are non-empty gns}} \\
   \hline
  $T - X$ & $T \cap X$ \\
    \hline
    \endfirsthead
    \hline
    \endhead
    \hline
    \endfoot
    \hline
    \endlastfoot
    $\big\{ \{1, 1\}, \{3, 3\}, \{6, 6\}, \{2, 2\}, \{4, 4\} \big\}$ & $\big\{ \{1, 2\}\big\}$ \\
    $\big\{ \{3, 3\}, \{6, 6\}, \{1, 1\}, \{4, 4\}, \{7, 7\} \big\}$ & \\
    \hline
    $\big\{ \{1, 1\}, \{4, 4\}, \{7, 7\}, \{2, 2\}, \{3, 3\} \big\}$ & $\big\{ \{1, 2\}, \{1, 3\}\big\}$ \\
   \hline 
    $\big\{ \{1, 1\}, \{3, 3\}, \{6, 6\}, \{2, 2\}, \{4, 4\} \big\}$ & $\big\{ \{1, 2\}, \{6, 7\}\big\}$ \\
    $\big\{ \{6, 6\}, \{1, 1\}, \{3, 3\}, \{2, 2\}, \{7, 7\} \big\}$ &   \\
    \hline
    $\big\{ \{2, 2\}, \{6, 6\}, \{4, 4\}, \{3, 3\}, \{5, 5\} \big\}$ & $\big\{ \{2, 5\}, \{3, 6\}, \{4, 7\}\big\}$ \\
    \hline
    $\big\{ \{2, 2\}, \{6, 6\}, \{4, 4\}, \{1, 1\}, \{7, 7\} \big\}$ & $\big\{ \{1, 2\}, \{1, 6\}, \{6, 7\}\big\}$ \\
    \hline
   $\big\{ \{2, 2\}, \{6, 6\}, \{4, 4\}, \{1, 1\}, \{7, 7\} \big\}$ & $\big\{ \{1, 2\}, \{1, 6\}, \{2, 7\}, \{6, 7\}\big\}$ \\ 
 \hline
\caption{\label{gnsLastcase10}  $T - X\subseteq X_E^*$, $|T - X |=5 $.}
\end{longtable}   
 
 These are generalised nice sets because  none of the 5 indices   belongs to $J_S$ (looking at Table~\ref{IsyJs}).
This time it is easier to discuss possible collineations.
If $|J_S|=2$, there is only one choice and nothing to discuss. So, we have only to study the three cases 
  with $|J_S|=1$. In all of them the generalised nice sets are $S_J$ for $J= I-\{5,k\}$.
  Let us begin with  
$S=\big\{ \{1, 2\}\big\}$. Here  one equivalence class is got from $k\in I_S-J_S=\{1,2\}$, and 
the other from $k\in I-I_S=\{3,4,6,7\}$, clearly giving two not collinear generalised nice sets.  
%
Something very similar happens for $S=\big\{ \{1, 2\}, \{6, 7\}\big\}$:  
two not collinear generalised nice sets arise,
   according to  $k\in I_S-J_S=\{1,2,6,7\}$, or $k\in I-I_S=\{3,4\}$.
In contrast, all $k\in \{2,3,4,5,6,7\}$ play exactly the same role in case $S=\big\{ \{2, 5\}, \{3, 6\}, \{4, 7\}\big\}$, which  only leads to one generalised nice set $S_{I-\{5,k\}}$. 
\smallskip

Let us consider the last two cases together: $T - X=\big\{ \{0, 0\}, \{0, i\}, \{0, j\}, \{0, k\}, \{0, l\} \big\},$ 
where $i, j, k, l\in I$ are distinct.

\begin{prop} \label{noolvidardel5}
If $S\subseteq X$ and $S'_J:=S\cup \big\{  \{0, 0\},\{0, j\}:j\in J\}$  are non-empty generalised nice sets such that $J\subseteq I$ has cardinal 4,
 then $|S|=1$ and $S'_J$ is collinear to one of these sets:
\begin{itemize}
\item[\rm(i)] $\big\{ \{0, 0\}, \{0, 3\}, \{0, 4\}, \{0, 6\}, \{0, 7\}, \{1, 2\}\big\};$
\item[\rm(ii)] $\langle P_{\{0, 1, 2\}} \rangle \cup \big\{\{0, 3\}\big\} = \big\{\{0, 0\}, \{0, 1\}, \{0, 2\}, \{0, 3\}, \{0, 5\}, \{1, 2\} \big\}.$
\end{itemize}
\end{prop}

\begin{proof} Recall from Theorem~\ref{todo} that   for any $\{a,b\}\in S$, we have
either $\ell_{ab}\cap J=\emptyset$, or $\ell_{ab}\subseteq J$.   

 First, assume   that any triplet among the elements in $ J$ is generative, that is, $J$ does not contain a line.
 Then $\ell_{ab}\cap J=\emptyset$ for any $\{a,b\}\in S$ and $I_S\subseteq I-J$. This implies $|I_S|\le| I-J|=3$, and hence $|S|=1$ by Table~\ref{IsyJs}.
 Assuming $S=\big\{\{1, 2\}\big\}$, the set in (i) is of course the only possibility.
 

 
 Second, consider the case where there is a line contained in $ J$. Then every line intersects $J$, and the condition above says $\ell_{ab}\subseteq J$
 for any $\{a,b\}\in S$. In other words, $I_S\subseteq J$ and $|I_S|\le|  J|=4$. Again Table~\ref{IsyJs} tell us that $|S|=1$. If $S=\big\{\{1, 2\}\big\}$, there is $i\in\{3,4,6,7\}$ such that $J=\ell_{12}\cup\{i\}$, but the four possibilities  {lead} to collinear sets.
\end{proof}


\subsection{$T - X$ of cardinality 6} 

Table \ref{gnsEmptyX} tells us that there are two types of generalised nice sets contained in $X_0 - X$, that are as follows:
\begin{itemize}
\item[-] $\big\{ \{i, i\}: i\in J \big\},$ for $J\subseteq I$, $|J|=6$;
\item[-] $\big\{ \{0, 0\}, \{0, i\}: i\in J \big\},$ for $J\subseteq I$, $|J|=5$.

\end{itemize}

Assume first that $T - X$ is as in the first case. Then one can prove the following:

\begin{cor} 
Let $S$ be a generalised nice set. For any $k\in I$,    $S_{I-\{k\}}$ is   generalised 
nice if and only if   $S\subseteq X^{(k)}$.   
\end{cor}

\begin{proof}
Note that saying that $S\subseteq X^{(k)}$ is the same as saying $J_S=\{k\}$. Theorem~\ref{todo} implies that $S_J$ is generalised if, and only if,
$J$ does not intersect $J_S$. Since $|J|=6$, this means $J=I-\{k\}.$
%
%
%
\end{proof}

In this way we obtain just three new generalised nice sets up to collineations (the two first sets contained in  $X^{(5)}$, to agree with the choices in Corollary~\ref{somegns}):
\smallskip 

\begin{longtable} [c] { | c | c|  }
    \hline
    \multicolumn{2} { | c | }{{\bf Gns $T$ such that both $T \cap X$ and $T - X$ are non-empty gns}} \\
   \hline
  $T - X$ & $T \cap X$ \\
    \hline
    \endfirsthead
    \hline
    $T - X$ & $T \cap X$ \\
    \hline
    \endhead
    \hline
    \endfoot
    \hline
    \endlastfoot
    $\big\{ \{1, 1\}, \{3, 3\}, \{6, 6\}, \{2, 2\}, \{4, 4\}, \{7, 7\} \big\}$ & $\big\{ \{1, 2\}\big\}$ \\
    \hline
    $\big\{ \{1, 1\}, \{3, 3\}, \{6, 6\}, \{2, 2\}, \{4, 4\}, \{7, 7\} \big\}$ & $\big\{ \{1, 2\}, \{6, 7\}\big\}$ \\
    \hline
     $\big\{ \{5, 5\}, \{3, 3\}, \{6, 6\}, \{2, 2\}, \{4, 4\}, \{7, 7\} \big\}$ & $X^{(1)}=\big\{ \{2, 5\}, \{3, 6\}, \{4, 7\}\big\}$ \\
 \hline
\caption{\label{gnsLastcase11}  $T - X\subseteq X_E^*$, $|T - X|=6$.
}
\end{longtable}

\smallskip

To finish, we consider the only alternative form of $T - X$ with cardinality 6: 

\begin{prop}\label{noolvidardel6}
If $S\subseteq X$ and $S'_J:=S\cup \big\{  \{0, 0\},\{0, j\}:j\in J\}$  are non-empty generalised nice sets such that $J\subseteq I$ has cardinal 5,
 then $|S|\le2$ and $S'_J$ is collinear to one of these sets:
\begin{itemize}
\item[\rm(i)] 
$\langle P_{\{0, 1, 2\}}\rangle \cup \big\{\{0, 6\}, \{0, 7\} \big\}  = \big\{\{0, 0\}, \{0, 1\}, \{0, 2\}, \{0, 5\}, \{0, 6\}, \{0, 7\}, \{1, 2\}\big\};$
\item[\rm(ii)]  
$\langle P_{\{0, 1, 2\}}\rangle \cup \big\{\{0, 6\}, \{0, 3\} \big\}  =  \big\{\{0, 0\}, \{0, 1\}, \{0, 2\}, \{0, 5\}, \{0, 3\}, \{0, 6\}, \{1, 2\}\big\}; $
\item[\rm(iii)] 
$\langle P_{\{0, 1, 2\}}\rangle \cup \langle P_{\{0, 6, 3\}}\rangle  =  \big\{\{0, 0\}, \{0, 1\}, \{0, 2\}, \{0, 5\}, \{0, 3\},\{0, 6\}, \{1, 2\}, \{1, 3\} \big\};$
\item[\rm(iv)] 
$\langle P_{\{0, 1, 2\}}\rangle \cup \langle P_{\{0, 6, 7\}}\rangle  =  \big\{\{0, 0\}, \{0, 1\}, \{0, 2\}, \{0, 5\}, \{0, 6\},\{0, 7\}, \{1, 2\}, \{6, 7\} \big\}.$
\end{itemize}
\end{prop}

\begin{proof}
Recall from Theorem~\ref{todo} that, as $|J|\ge5$, 
 we have a necessary and sufficient condition on $J$ to guarantee $S'_J$ generalised nice: $I_S\subseteq J$. 
This implies $|I_S|\le5$,  so that $|S|\le2$ according to Table~\ref{IsyJs}.

In case $|S|=2$, as $|I_S|=5$ then necessarily $I_S=J$ and cases (iii) and (iv) clearly arise as the unique possibilities.
 In case $|S|=1$, we have to study possible $J$'s containing $I_S$  with 5 elements. 
 If we fix $S=\big\{\{1,2\}\big\}$, then
 $J=\ell_{12}\cup\ell_{ab}$. We have to consider two situations: either $\ell_{12}\cap\ell_{ab}$ equals one of the two indices appearing in $S$, or
 $\ell_{12}\cap\ell_{ab}=\{5\}$.  Evidently, these possibilities are not collinear.
\end{proof}

\subsection{$T - X$ of cardinality 7} 

Looking at Table \ref{gnsEmptyX} we get the following generalised nice sets contained in $X_0 - X$:
\begin{itemize}
\item[-] $X_E^*=\big\{ \{1, 1\}, \{2, 2\}, \{3, 3\}, \{4, 4\}, \{5, 5\}, \{6, 6\}, \{7, 7\}\big\};$
\item[-] $\big\{ \{0, 0\}, \{0, i\}, i\in J \big\},$ for $J\subseteq I$, $|J|=6$.   
\end{itemize}

In the first case, $S\cup X_E^*=S_I$ is never a   generalised nice set, by Theorem~\ref{todo}(a) (of course $I$ intersects $J_S$).
We consider the only alternative for $T - X$ with cardinality 7, to get, from Theorem~\ref{todo}(ii),
\begin{cor}
Let  $S\subseteq X$ be a  non-empty generalised nice set, and $k\in I$. Then $S'_{I-\{k\}}$ is a generalised nice  set if and only if $k\notin I_S$.
\end{cor}
In particular,    $S$   cannot be any of the sets such that $I_S=I$, that is, those in Eq.~\eqref{losdeI_S=I}. 
Any other $S$ can be combined with $T - X=\big\{ \{0, 0\}, \{0, i\}: i\ne k \big\},$ for any choice of $k\notin I_S.$  
We obtain

\begin{longtable} [c] { | c | c|  }
    \hline
    \multicolumn{2} { | c | }{{\bf Gns $T$ such that both $T \cap X$ and $T - X$ are non-empty gns}} \\
   \hline
  $T - X$ & $T \cap X$ \\
    \hline
    \endfirsthead
    \hline
    \endhead
    \hline
    \endfoot
    \hline
    \endlastfoot
    $\big\{ \{0, 0\}, \{0, 1\}, \{0, 2\}, \{0, 5\}, \{0, 3\}, \{0, 4\}, \{0, 6\}\big\}$ & $\big\{ \{1, 2\}\big\}$ \\
    \hline
    $\big\{ \{0, 0\}, \{0, 1\}, \{0, 2\}, \{0, 5\}, \{0, 3\}, \{0, 4\}, \{0, 6\}\big\}$ & $\big\{ \{1, 2\}, \{1, 3\}\big\}$ \\
    $\big\{ \{0, 0\}, \{0, 1\}, \{0, 2\}, \{0, 5\}, \{0, 3\}, \{0, 7\}, \{0, 6\}\big\}$ &  \\
   \hline 
   $\big\{ \{0, 0\}, \{0, 1\}, \{0, 2\}, \{0, 5\}, \{0, 3\},  \{0, 6\},   \{0, 7\}\big\}$ & $\big\{ \{1, 2\}, \{6, 7\}\big\}$ \\
    \hline
    $\big\{ \{0, 0\}, \{0, 1\}, \{0, 2\}, \{0, 5\}, \{0, 3\}, \{0, 4\}, \{0, 6\}\big\}$ & $\big\{ \{1, 2\}, \{1, 6\}, \{2, 6\}\big\}$ \\
    \hline
    $\big\{ \{0, 0\}, \{0, 1\}, \{0, 2\}, \{0, 5\}, \{0, 3\}, \{0, 6\},   \{0, 7\}\big\}$ & $\big\{ \{1, 2\}, \{1, 6\}, \{6, 7\}\big\}$ \\
    \hline
   $\big\{ \{0, 0\}, \{0, 1\}, \{0, 2\}, \{0, 5\}, \{0, 3\}, \{0, 6\}, \{0, 7\}\big\}$ & $\big\{ \{1, 2\}, \{1, 6\}, \{2, 7\}, \{6, 7\}\big\}$ \\
    \hline
  \caption{\label{gnsLastcase12} $T - X \subseteq X_F$, $|T-X|=6$}
\end{longtable}

We only need to be concerned about those sets with $|I_S|<6$, corresponding with $|S|\le 2$, where there are several choices for $k$ and we have to decide if they produce collinear sets when combining $S$ and $T-X$. For
$\big\{ \{1, 2\}, \{6, 7\}\big\}$, $k$ could be either $3$ or $4$, but they give  collinear sets (for instance, $\sigma_{124}$ is a suitable collineation). 
On the contrary, 
if $S=\big\{ \{1, 2\}, \{1, 3\}\big\}$, $k$ could be either $4$ or $7$, but the related generalised nice sets are not collinear: a such collineation would have to preserve $\{1\}$ and $\{2,3\}$, hence $\{7\}$.


\subsection{$T - X$ of cardinality 8}

There is only one generalised nice set contained in $X_0 - X$ of cardinality 8, namely, 
$
X_F=\big\{\{0, 0\}, \{0, i\} :i\in I\big\}.
$ From Theorem~\ref{todo}(a), it follows 

\begin{cor}\label{gnsLastcase13}
For any generalised nice set $S$ given in Corollary~\ref{somegns}, $S\cup X_F$ is   a  generalised nice set too.
\end{cor}


\section{Computer Assistance} \label{computer}

In this section, we describe some algorithms that have enabled us to verify the above classification in an additional alternative way. 
The first returns a set consisting of all the collineations which act as equivalences 
between two generalised nice sets. 
The second checks whether a given set is generalised nice. 
    
\subsection{Collineation Search}
  
Here, we describe an algorithm which, given two generalised nice sets $S$ and $T,$ identifies all collineations $\sigma \in S_\ast(I)$ which act as $\tilde \sigma(S)=T$ or $\tilde 
  \sigma (T)= S.$ We begin by introducing some definitions.
  
\smallskip
  
To ease the notation, in what follows, we will refer to the elements of $X_0$ as pairs.
  
    \begin{define} Fix $i\in I_0,$ $S$ a generalised nice set, and $p, q$ pairs in $S.$ 
    \begin{itemize}
\item[\rm (i)] $\vert \setbar{ \{a, b\}\in S}{a\ast b=i}\vert$ is called the \emph{height} of $i$ in $S.$ 
\item[\rm (ii)] The number of occurrences of $i$ in the pairs in $S$ (that is, up to twice per pair) is called the \emph{weight} of $i$ in $S.$
\item[\rm (iii)] For $0 \leq n \leq 9$, the subset $\freq(n, S)$ of $I$ consisting of all the elements of weight $n$ in $S$ is called the \emph{frequency set} of weight $n$ in $S.$ 
\item[\rm (iii)]We say that $\setbar{a \ast b}{\{a,b\}\in S}$ is the set of \emph{points} in $S.$
    \end{itemize}
We define the actions of a collineation $\sigma$ on: 
    \begin{itemize}
    \item[\rm (iv)] $I$ given by $\bar \sigma(J)= \setbar{\sigma(j)}{j\in J}$ for all $J\subseteq I,$ and
    \item[\rm (v)] $X_0$ given by $\tilde \sigma(S) = \left\{\{\sigma(s_1), \sigma(s_2)\} \vert \{s_1, s_2\}\in S\right\}.$
    \end{itemize}
    \end{define}

This following result will be very useful when searching for collineations between generalised nice sets. The proof is straightforward and we omit it here.
    
\begin{lemma} \label{eFreq}
The actions $\ \bar{} \ $ and $\ \tilde{} \ $ preserve frequency sets: $\bar\sigma\left(\freq(n,S)\right)= 
    \freq(n, \tilde \sigma(S)),$ for all generalised nice sets $S$ and $\sigma \in S_\ast(I)$. Moreover, if $S$ and $T$ are collinear generalised nice sets, then $|\freq(n,S)| = |\freq(n,T)|$ for all $0 \leq n \leq 9$.  
\end{lemma}

Lemma~\ref{eFreq} reduces our search to collineations $\sigma$ which obey $\tilde \sigma \left(\freq(n,S)\right)= \freq(n, T)$ for all $0 \leq n \leq 9.$ 
Moreover, each pair of frequency sets $\freq(n,S)$ and $\freq(n,T)$ provides a necessary condition on $\sigma$ which depends on the cardinality of the given frequency set. More precisely, for $\freq(n,S) =\{s_1, \ldots, s_c\}$ and $\freq(n,T) =\{t_1, \ldots, t_c\}$ we illustrate this process: 

    \begin{enumerate}
    		\item[\tiny $(c=1):$] In the singleton case we must have $\sigma(s_1)=t_1.$
		\item[\tiny $(c=2):$]  In this case, we find $\sigma(s_1 \ast s_2)=t_1\ast t_2.$
		\item[\tiny $(c=3):$] We have two possibilities:
		\begin{itemize}
		\item[-] If $s_1, s_2, s_3$ are generative then $\sigma(s_1\ast s_2 \ast s_3)= t_1 \ast t_2 \ast t_3.$
		\item[-]Otherwise, 
		$s_1, s_2, s_3$ form a line $L_S$ which $\sigma$ must send to $L_T,$ the analogous line for $T.$ That is, $\sigma(L_S)=L_T.$
		\end{itemize}
		\item[\tiny $(c=4):$] Instead of considering $\freq(n,S)$ and $\freq(n,T)$ we consider $I - \freq(n,S)$ and $I-\freq(n,T).$ This leaves us in the case $c=3.$
		\item[\tiny $(c=5):$] Again, we consider $I-\freq(n,S)$ and $I-\freq(n,T)$ and we are in the case $c=2.$
		\item[\tiny $(c=6):$] We consider $I-\freq(n,S)$ and $I-\freq(n,T),$ each of which are singletons. We are in the case $c=1.$
    \end{enumerate} 
Next, we consider the set consisting of all collineations adhering to all the conditions above.
If there exist any collineations mapping $S$ onto $T,$ they must belong to this set.

\begin{remark} \label{complementEx}
	Take collinear generalised nice sets $S$ and $T$. Then we have a disjoint union $I= \bigcup_{0\leq n \leq 9} \freq(n, S).$ Furthermore, if this disjoint union consists of two 
	non-empty frequency sets (necessarily of different weights) of $S,$ then these two frequency sets each offer us the same information about collineations.
	
	For example, let us suppose that $S=\{\{1,1\}, \{2,2\}\}$ and $T=\{\{2,2\}, \{3,3\}\}.$ Then $I= \freq(0,S)\cup \freq(2, S)=  \{3,4,5,6,7\} \cup \{1,2\}$ and 
	$I=\freq(0, T)\cup \freq(2, T)=\{1,4,5,6,7\}\cup \{2,3\}.$ Both of these respective pairs of sets offer us the same necessary condition; namely, $\sigma(1\ast 2)= 2\ast 3,$ for any 
	$\sigma\in S_\ast(I)$ acting via $\tilde{} \ $ (as an equivalence between $S$ and $T$). 

\smallskip
	
A similar situation occurs whenever there is a disjoint union $I=\freq(n_1, S)\cup \freq(n_2, S),$ 
	with respective cardinalities $c_1$ and $c_2= 7-c_1,$ such that $1\leq c_1 \leq 6.$
\end{remark}

\subsection{Variables used}

     Here we describe the variables used the algorithms. The variable {\tt sigma} will be used to track the information we have found regarding the necessary 
     conditions for collineations between $S$ and $T.$ In fact, {\tt sigma} is a list with seven elements. For $1\leq i \leq 7,$ if the value of {\tt sigma[i]} is 
\begin{itemize}
\item[-] non-zero, then this indicates that $\sigma(i)$ is equal to the value of {\tt sigma[i]};
\item[-] 0, then this indicates that we are not yet sure of the value of $\sigma(i).$
\end{itemize}     

 The five variables {\tt used2, used31, used32, used4,} and {\tt used5} are used to track whether a complementary subset (as per in Remark \ref{complementEx}) exists amongst the frequency sets of $S.$ Note that there are two types of frequency sets of cardinality 3: lines or generative triplets. These types correspond to {\tt used31} and {\tt used32}, respectively.

We find  it convenient to implement a generalised nice set $S$ as a list of ordered pairs. To be more precise, these lists take the form
\begin{equation}
S=[[s_{11}, s_{12}], \ldots, [s_{m1}, s_{m2}]],
\end{equation}
where each $s_{ij} \in I_0.$ Furthermore, we will assume the following properties:
\begin{itemize} \label{sorted_properties}
	\item[-] $s_{i1}\leq s_{i2}$ for all $1\leq i \leq 7,$
	\item[-] $s_{i1} \leq s_{(i+1) 1}$ for all $1\leq i \leq 6,$	
	\item[-] $s_{i2} \leq s_{(i+1) 2}$ for all $1\leq i \leq 6.$ 
\end{itemize}
In practice, 'sorting' generalised nice sets (realised as lists of ordered pairs) is straightforward. This achieves the properties described above.

\subsection{Pseudo Code}

In this short section, we provide pseudo code for the algorithms used. We have used the mathematics software \emph{Maple} in our implementation of each of the following algorithms.

\subsubsection{Finding collineations} 

The following algorithm describes a method to construct the set of all collineations between two given generalised nice sets.

\smallskip

    \begin{tcolorbox}[sharp corners, boxrule=0.5pt, breakable]
        \hfill \\    
    \textit{Input:} Two generalised nice sets $S$ and $T$ with $|S|=m$ and $|T|=n.$ \\
    \textit{Output:} The set $C$ of all collinear permutations under whose action $S$ and $T$ are collinear.

%
%
    
		\begin{enumerate}[label=\arabic*.]
			\item 
				\begin{enumerate}
					\item $C\leftarrow \emptyset.$
					
					\item \textbf{If} $|S| \neq |T|$ \textbf{or} $| \setbar{i\in I}{\exists \{i,j\} \in S}| \neq  | \setbar{i\in I_0}{\exists \{i,j\} \in T}|$ \textbf{or} $|\text{freq}(0,S)| \neq |\text{freq}
					(0,T)|,$ \textbf{then return} $\emptyset$ \textbf{fi:}
			
					\item {\textbf{For} $1 \leq {\tt i} \leq 10$ \textbf{do} \newline
							\hspace*{6pt} \textbf{If} $|\freq({\tt i},S)| \neq |\freq({\tt i},T)|,$ \textbf{then return} $\emptyset$ \textbf{fi:} 
\\
	 						\textbf{od:}} 

%
%
					\item \textbf{If} $\vert \setbar{a\ast b}{\{a,b\} \in S}\vert \neq \vert \setbar{a\ast b}{\{a, b\}\in T} \vert$  \textbf{then return} $\emptyset$ \textbf{fi:} 
				\end{enumerate}
			
			\item 
			\begin{enumerate}
				\item ${\tt sigma} \leftarrow [0,0,0,0,0,0,0].$
				\item  {\tt used2, used31, used32, used4, used5} $\leftarrow$ false.  
				\item {\tt used3Count} $\leftarrow1.$
			\end{enumerate}
				\item \textbf{For} $1 \leq {\tt i} \leq 10$ \textbf{do}
			\begin{enumerate}
				\item {\tt case 1:} $|\text{freq}({\tt i}, S)|=1$ \newline
				\textbf{If} ${\tt sigma}[\text{freq}({\tt i},S)[1]]=0$ {\tt and} $\text{freq}({\tt i},T)[1] \neq 0$ \textbf{then} \hspace*{4pt} ${\tt sigma}[\text{freq}({\tt i},S)[1]]\leftarrow 
				\freq({\tt i},T)[1].$ 
				\newline \hspace*{4pt} \textbf{If} {\tt sigma} has exactly two non-zero elements, ${\tt sigma}[m]$ and ${\tt sigma}[n],$  \hspace*{12pt} \textbf{then} ${\tt sigma}[m\ast n] 
				\leftarrow {\tt sigma}[m]\ast {\tt sigma}[n].$\newline 
				\hspace*{4pt} \textbf{fi:} \newline
				\textbf{fi:}
				
				\item {\tt case 2:} $|\text{freq}({\tt i}, S)|=2$ \textbf{and not} {\tt used5}  \newline
						\hspace*{4pt} ${\tt p} \leftarrow \freq({\tt i}, S)[1] \ast \freq({\tt i}, S)[2], \ {\tt q} \leftarrow \freq({\tt i}, T)[1] \ast \freq({\tt i}, T)[2]$ \textbf{fi:}  \newline 
						\hspace*{4pt} \textbf{If} {\tt sigma}$[{\tt p}]=0$ \textbf{and} ${\tt p} \neq 0$ \textbf{and} ${\tt p}\not\in \freq({\tt i},S)$ \textbf{then} {\tt sigma}$[{\tt p}]\leftarrow{\tt q}.$
						\hspace*{8pt}\textbf{fi:}
						\newline \hspace*{4pt} \textbf{If} {\tt sigma} has exactly two non-zero elements, ${\tt sigma}[m]$ and ${\tt sigma}[n]$ \hspace*{6pt} \hspace*{4pt}\textbf{then} 
						${\tt sigma}[m\ast n] \leftarrow {\tt sigma}[m]\ast {\tt sigma}[n].$\newline
						\hspace*{4pt} \textbf{fi:} \newline 
						\textbf{od:}
						
				\item {\tt case 3:} $|\text{freq}({\tt i}, S)|=3$ \textbf{and not} {\tt used4}  \newline
						\hspace*{4pt} \textbf{If} exactly one of $\freq({\tt i}, S)$ or $\freq({\tt i},T)$ constitutes a line \newline
						\hspace*{8pt} \textbf{then return} $\emptyset.$ \newline
						\hspace*{4pt} \textbf{fi:} \newline
						\hspace*{4pt} \textbf{If} $\freq({\tt i}, S)$ is a generative triplet \textbf{then} \newline 
						\hspace*{8pt} ${\tt p} \leftarrow \freq({\tt i}, S)[1] \ast \freq({\tt i}, S)[2] \ast \freq({\tt i}, S)[3].$ \newline
						\hspace*{8pt} ${\tt q} \leftarrow \freq({\tt i}, T)[1] \ast \freq({\tt i}, T)[2] \ast \freq({\tt i}, T)[3].$
						\newline \hspace*{8pt} \textbf{If} {\tt sigma}$[{\tt p}]=0$ \textbf{and} ${\tt q}\notin \{0\} \cup \text{freq}({\tt i},T)$ \textbf{and} $0\notin \freq({\tt i},T)$ \textbf{then} 
						\newline \hspace*{12pt} ${\tt sigma[p]}\leftarrow {\tt sigma[q]}.$ \newline
						\hspace*{8pt} \textbf{fi:}
						\newline \hspace*{8pt} \textbf{If} {\tt sigma} has exactly two non-zero elements,${\tt sigma}[m]$ and ${\tt sigma}[n],$ \hspace*{8pt} \hspace*{4pt} \textbf{then} 
						${\tt sigma}[m\ast n] \leftarrow {\tt sigma}[m]\ast {\tt sigma}[n].$ \newline
						\hspace*{8pt} \textbf{fi:}
						\newline \hspace*{4pt} \textbf{else} $\text{all}3[\text{used3Count}]\leftarrow \setbar{\theta\in S_\ast}{\theta(L)=L}.$ 
						\newline \hspace*{8pt} {\tt used3}$[${\tt used3Count}$] \leftarrow$ true,  {\tt used3Count}$ \leftarrow ${\tt used3Count}$+1.$ \newline
						\hspace*{4pt} \textbf{fi:}
						
				\item $|\text{freq}(i,S)|=4$ \textbf{and not} {\tt used3}$[1]$  \newline  
				\hspace*{4pt} $J \leftarrow I\setminus \text{freq}(i,S).$ \newline
				\hspace*{4pt} follow {\tt case 3} with $J.$ \newline
				\hspace*{4pt} \textbf{If} $J$ forms a line $L$ \textbf{then} \newline
				\hspace*{8pt} ${\tt all}4 \leftarrow \setbar{\theta\in S_\ast}{\theta(L)=L}.$ \newline
				\hspace*{8pt} {\tt used}$4 \leftarrow$ true.
						
				\item $|\text{freq}(i,S)|=5$ \textbf{and not} {\tt used2} \newline
				\hspace*{4pt} $J \leftarrow I\setminus \text{freq}(i,S).$ \newline
				\hspace*{4pt} follow {\tt case 2} with $J.$ \newline
				\hspace*{4pt} {\tt used5} =true.
				
				\item $|\text{freq}(i,S)|=6$ \newline
				\hspace*{4pt} $ J\leftarrow I\setminus \text{freq}(i,S)$. \newline
				\hspace*{4pt} follow {\tt case 1} with $J.$
			\end{enumerate}
			 
			\item \textbf{If} the information in {\tt sigma} uniquely defines a collineation $\sigma$ \textbf{then return}  $\{\sigma\}$. 
			
			\item $A \leftarrow \{\sigma \in S_\ast(I) \vert \ \sigma \text{ obeys all the information in {\tt sigma}\}}. $
			\item 
			\begin{enumerate}
				\item  \textbf{If} {\tt used3[1]}= true \textbf{then} $A \leftarrow A \cup {\tt all}3[1]$ \textbf{fi:}
				
				\item \textbf{If} {\tt used3[2]}= true \textbf{then} $A \leftarrow A \cap {\tt all}3[2]$ \textbf{fi:}
				
				\item \textbf{If} {\tt used4} \textbf{then} $A \leftarrow A\cap \text{all}4$ \textbf{fi:}
			\end{enumerate} 
			   
			\item \textbf{For each} $\theta\in A$ \textbf{do} \newline
			\hspace*{4pt} \textbf{If} $S$ and $T$ are collinear under $\theta$ \textbf{then} $C\leftarrow C \cup \{\theta\}$ \textbf{fi:}
			\item \textbf{Return} $C.$
		\end{enumerate}
	\end{tcolorbox}
	
	\subsubsection{Generalised Nice Verification}
	We now describe an algorithm which checks whether a given subset of $X_0$ is a generalised nice set. 
	
\smallskip
	
	 \begin{tcolorbox}[sharp corners, boxrule=0.5pt, breakable]
          \hfill \\    
    \textit{Input: A subset $S\subseteq X_0.$} \\
    \textit{Output: True or false indicating whether $S$ is generalised nice.} 
		\begin{enumerate}[label=\arabic*.]
			
			\item {\tt check} $ \leftarrow$ false, {\tt bank} $\leftarrow \emptyset.$
			
			\item \textbf{For} ${\tt i}$ \textbf{from} $1$ \textbf{to} $|S|-1$ \textbf{do} \newline
			 \hspace*{4pt} \textbf{For} ${\tt j}$ \textbf{from} $ {\tt i}+1$ \textbf{to} $|S|$ \textbf{do} \newline
			 \hspace*{8pt} ${\tt c} \leftarrow S[{\tt i}][1] \ast S[{\tt i}][2].$ \newline
			 \hspace*{8pt} \textbf{If} ${\tt c}=S[{\tt j}][1]$ \textbf{then} \newline
			 \hspace*{12pt} ${\tt a} \leftarrow S[{\tt i}][1],\ {\tt b} \leftarrow S[{\tt i}][2], \ {\tt c} \leftarrow S[{\tt j}][2],$ {\tt check} $\leftarrow $ true. \newline
			 \hspace*{8pt} \textbf{Elif} ${\tt c}= S[{\tt j}][2]$ \textbf{then} \newline 
			 \hspace*{12pt} ${\tt a} \leftarrow S[{\tt i}][1],\ {\tt b} \leftarrow S[{\tt i}][2], \ {\tt c} \leftarrow S[{\tt j}][1],$ {\tt check} $\leftarrow$ true. \newline
			 \hspace*{8pt} \textbf{fi:} \newline
			 \hspace*{8pt} \textbf{If} {\tt check} =false \textbf{then} ${\tt c} \leftarrow S[{\tt j}][1] \ast S[{\tt j}][2]$ \newline
			 \hspace*{12pt} \textbf{If} ${\tt c}=S[{\tt i}][2]$ \textbf{then} \newline
			 \hspace*{16pt} ${\tt a} \leftarrow S[{\tt j}][1],\ {\tt b} \leftarrow S[{\tt j}][2], \ {\tt c} \leftarrow S[{\tt i}][1],$ {\tt check} $\leftarrow  $ true. \newline
			 \hspace*{12pt} \textbf{Elif} ${\tt c}=S[{\tt i}][1]$ \textbf{then} \newline
			 \hspace*{16pt} ${\tt a} \leftarrow S[{\tt j}][1],\ {\tt b} \leftarrow S[{\tt j}][2], \ {\tt c} \leftarrow S[{\tt i}][2],$ {\tt check} $\leftarrow$ true.  \newline
			 \hspace*{12pt} \textbf{fi:} \newline
			 \hspace*{8pt} \textbf{fi:} \newline
			 \hspace*{8pt} \textbf{If} {\tt check} = true \textbf{and} $\{a,b,c\} \notin $ {\tt bank} \textbf{then} \newline
			 \hspace*{12pt} \textbf{If} $P_{\{a,b,c\}}\in S$ \textbf{then} {\tt bank} $\leftarrow$  {\tt bank} $\cup \ \{a,b,c\}.$ \newline
			 \hspace*{12pt} \textbf{Else return} false. \newline
			 \hspace*{12pt} \textbf{fi:} \newline
			 \hspace*{8pt} \textbf{fi:} \newline
			 \hspace*{4pt} \textbf{od:} \newline
			 \textbf{od:}
			 \item \textbf{Return} true.
		\end{enumerate}
	\end{tcolorbox}
	

\section{Conclusion} \label{conclusion}

Here we provide a brief summary of the results achieved and the next step in this research project. 
 
Let $T$ denote a   generalised nice set. Up to collineations, we found a total of 245.
More precisely, there are the empty-set and:
\begin{enumerate}
\item[-] 13 with $T$ contained in $X$,     detailed in Corollary~\ref{somegns};    
\item[-] 20 with $T$ having an empty intersection with $X$,   described in Table~\ref{gnsEmptyX};  
\item[-] 7 such that 
$T \cap X$ is not generalised nice;   appearing in Theorem~\ref{GnsListThm};
\item[-] 204 satisfying that $T \cap X \neq \emptyset$ and $T \cap X$ is generalised nice:  
 
\begin{itemize}
\item With $|T - X|=1$: 13 in Table~\ref{gnsLastcase00}, and 27 in Table~\ref{gnsLastcase}; 
\item With $|T - X|=2$: 41 in Table~\ref{gnsLastcase2}, and 7  in Table~\ref{gnsLastcase3};
\item With $|T - X|=3$: 33 in Table~\ref{gnsLastcase4},    8  in Table~\ref{gnsLastcase5},   4  in Table~\ref{gnsLastcase6}, and 7  in Table~\ref{gnsLastcase7}; 
\item With $|T - X|=4$: 13 in Table~\ref{gnsLastcase8}, 12 in Table~\ref{gnsLastcase9}, and 2 in Proposition~\ref{noolvidardel4};
\item With $|T - X|=5$: 8 in Table~\ref{gnsLastcase10}, and 2 in Proposition~\ref{noolvidardel5};
\item With $|T - X|=6$: 3 in Table~\ref{gnsLastcase11}, and 4 in Proposition~\ref{noolvidardel6};
\item With $|T - X|=7$: 7 in Table~\ref{gnsLastcase12};
\item With $|T - X|=8$: 13 in Corollary~\ref{gnsLastcase13}.
\end{itemize}
\end{enumerate}
 
This classification has allowed us in \cite{Paper4}   to find some   new families of Lie algebras obtained by  graded contractions 
 of suitable 
$\mathbb Z_2^3$-gradings on exceptional Lie algebras. 
The role that the notion of a generalised nice set has played in this research is similar to that of a nice set in \cite[\S 5]{Paper1}.

\section*{Acknowledgements} 

The first and third authors were supported by Junta de Andaluc\'{\i}a through project FQM-336. 
The first author was supported by the Spanish Ministerio de Ciencia e Innovaci\'on through projects  
PID2020-118452GB-I00 and PID2023-152673NB-I00, all of them with FEDER funds. The second author also 
received a University of Cape Town Science Faculty PhD Fellowship and the Harry Crossley Research
Fellowship. The third author is supported by an URC fund 459269.

Part of this research was undertaken while the second author visited the Department of Algebra, Geometry and Topology of the University of M\'alaga in the fall of 2022. He thanks the department for their hospitality.    
\smallskip

\subsection*{Data availability}
No data was used for the research described in the article.

\end{document}